\documentclass[12pt]{amsart}
\usepackage{amsmath}
\usepackage{amsxtra}
\usepackage{amscd}
\usepackage{amsthm}
\usepackage{amsfonts}
\usepackage{amssymb}
\usepackage{eucal}
\usepackage{verbatim}
\usepackage[all]{xy}
\usepackage{color}
\definecolor{deepjunglegreen}{rgb}{0.0, 0.29, 0.29}
\definecolor{darkspringgreen}{rgb}{0.09, 0.45, 0.27}
\definecolor{Red}{rgb}{0.7, 0,0}

\makeatletter
\usepackage{etex,etoolbox,needspace}
\pretocmd\section{\Needspace*{4\baselineskip}}{}{}
\makeatletter

\usepackage[pdftex,bookmarks=false,colorlinks=true,citecolor=darkspringgreen,debug=true,
  naturalnames=true,pdfnewwindow=true]{hyperref}

\newtheorem{thm}{Theorem}[subsection]

\newtheorem{cor}[thm]{Corollary}

\newtheorem{lem}[thm]{Lemma}
\newtheorem{prop}[thm]{Proposition}
\newtheorem{conj}[thm]{Conjecture}

\theoremstyle{definition}

\theoremstyle{remark}
\newtheorem{rem}[thm]{Remark}

\newcommand{\nc}{\newcommand}
\nc{\renc}{\renewcommand} \nc{\ssec}{\subsection}
\nc{\sssec}{\subsubsection}

\nc{\on}{\operatorname} \nc{\wh}{\widehat}
\nc\ol{\overline} \nc\ul{\underline} \nc\wt{\widetilde}

\newcommand{\red}[1]{{\color{Red}#1}}

\nc{\BA}{{\mathbb{A}}} \nc{\BC}{{\mathbb{C}}} \nc{\BF}{{\mathbb{F}}}
\nc{\BD}{{\mathbb{D}}} \nc{\BG}{{\mathbb{G}}} \nc{\BQ}{{\mathbb{Q}}}
\nc{\BM}{{\mathbb{M}}} \nc{\BN}{{\mathbb{N}}} \nc{\BO}{{\mathbb{O}}}
\nc{\BP}{{\mathbb{P}}} \nc{\BR}{{\mathbb{R}}}
\nc{\BZ}{{\mathbb{Z}}} \nc{\BS}{{\mathbb{S}}} \nc{\BW}{{\mathbb{W}}}

\nc{\CA}{{\mathcal{A}}} \nc{\CB}{{\mathcal{B}}} \nc{\CalC}{{\mathcal{C}}} \nc{\CalD}{{\mathcal{D}}}
\nc{\CE}{{\mathcal{E}}} \nc{\CF}{{\mathcal{F}}}
\nc{\CG}{{\mathcal{G}}} \nc{\CH}{{\mathcal{H}}}
\nc{\CI}{{\mathcal{I}}} \nc{\CK}{{\mathcal{K}}} \nc{\CL}{{\mathcal{L}}}
\nc{\CM}{{\mathcal{M}}} \nc{\CN}{{\mathcal{N}}}
\nc{\CO}{{\mathcal{O}}} \nc{\CP}{{\mathcal{P}}}
\nc{\CQ}{{\mathcal{Q}}} \nc{\CR}{{\mathcal{R}}}
\nc{\CS}{{\mathcal{S}}} \nc{\CT}{{\mathcal{T}}}
\nc{\CU}{{\mathcal{U}}} \nc{\CV}{{\mathcal{V}}}  \nc{\CY}{{\mathcal Y}}
\nc{\CW}{{\mathcal{W}}} \nc{\CZ}{{\mathcal{Z}}}

\nc{\cM}{{\check{\mathcal M}}{}} \nc{\csM}{{\check{\mathcal A}}{}}
\nc{\oM}{{\overset{\circ}{\mathcal M}}{}}
\nc{\obM}{{\overset{\circ}{\mathbf M}}{}}
\nc{\oCA}{{\overset{\circ}{\mathcal A}}{}}
\nc{\obA}{{\overset{\circ}{\mathbf A}}{}}
\nc{\ooM}{{\overset{\circ}{M}}{}}
\nc{\osM}{{\overset{\circ}{\mathsf M}}{}}
\nc{\vM}{{\overset{\bullet}{\mathcal M}}{}}
\nc{\nM}{{\underset{\bullet}{\mathcal M}}{}}
\nc{\oD}{{\overset{\circ}{\mathcal D}}{}}
\nc{\obD}{{\overset{\circ}{\mathbf D}}{}}
\nc{\oA}{{\overset{\circ}{\mathbb A}}{}}
\nc{\op}{{\overset{\bullet}{\mathbf p}}{}}
\nc{\cp}{{\overset{\circ}{\mathbf p}}{}}
\nc{\oU}{{\overset{\bullet}{\mathcal U}}{}}
\nc{\ofZ}{{\overset{\circ}{\mathfrak Z}}{}}

\nc{\fa}{{\mathfrak{a}}} \nc{\fb}{{\mathfrak{b}}}
\nc{\fd}{{\mathfrak{d}}} \nc{\fe}{{\mathfrak{e}}} \nc{\ff}{{\mathfrak{f}}}
\nc{\fg}{{\mathfrak{g}}} \nc{\fgl}{{\mathfrak{gl}}}
\nc{\fh}{{\mathfrak{h}}} \nc{\fri}{{\mathfrak{i}}}
\nc{\fj}{{\mathfrak{j}}} \nc{\fk}{{\mathfrak{k}}} \nc{\fl}{{\mathfrak{l}}}
\nc{\fm}{{\mathfrak{m}}} \nc{\fn}{{\mathfrak{n}}}
\nc{\ft}{{\mathfrak{t}}} \nc{\fu}{{\mathfrak{u}}} \nc{\fv}{{\mathfrak{v}}}
\nc{\fw}{{\mathfrak{w}}} \nc{\fz}{{\mathfrak{z}}}
\nc{\fp}{{\mathfrak{p}}} \nc{\fq}{{\mathfrak{q}}} \nc{\frr}{{\mathfrak{r}}}
\nc{\fs}{{\mathfrak{s}}} \nc{\fsl}{{\mathfrak{sl}}}
\nc{\fso}{{\mathfrak{so}}} \nc{\fsp}{{\mathfrak{sp}}} \nc{\osp}{{\mathfrak{osp}}}
\nc{\hsl}{{\widehat{\mathfrak{sl}}}}
\nc{\hgl}{{\widehat{\mathfrak{gl}}}}
\nc{\hg}{{\widehat{\mathfrak{g}}}}
\nc{\chg}{{\widehat{\mathfrak{g}}}{}^\vee}
\nc{\hn}{{\widehat{\mathfrak{n}}}}
\nc{\chn}{{\widehat{\mathfrak{n}}}{}^\vee}

\nc{\fA}{{\mathfrak{A}}} \nc{\fB}{{\mathfrak{B}}} \nc{\fC}{{\mathfrak{C}}}
\nc{\fD}{{\mathfrak{D}}} \nc{\fE}{{\mathfrak{E}}}
\nc{\fF}{{\mathfrak{F}}} \nc{\fG}{{\mathfrak{G}}} \nc{\fH}{{\mathfrak{H}}}
\nc{\fI}{{\mathfrak{I}}} \nc{\fJ}{{\mathfrak{J}}}
\nc{\fK}{{\mathfrak{K}}} \nc{\fL}{{\mathfrak{L}}}
\nc{\fM}{{\mathfrak{M}}} \nc{\fN}{{\mathfrak{N}}}
\nc{\frP}{{\mathfrak{P}}} \nc{\fQ}{{\mathfrak{Q}}}
\nc{\fS}{{\mathfrak{S}}} \nc{\fT}{{\mathfrak{T}}} \nc{\fU}{{\mathfrak{U}}}
\nc{\fV}{{\mathfrak{V}}} \nc{\fW}{{\mathfrak{W}}}
\nc{\fX}{{\mathfrak{X}}} \nc{\fY}{{\mathfrak{Y}}}
\nc{\fZ}{{\mathfrak{Z}}}

\nc{\ba}{{\mathbf{a}}}
\nc{\bb}{{\mathbf{b}}} \nc{\bc}{{\mathbf{c}}} \nc{\be}{{\mathbf{e}}}
\nc{\bg}{{\mathbf{g}}} \nc{\bj}{{\mathbf{j}}} \nc{\bm}{{\mathbf{m}}}
\nc{\bn}{{\mathbf{n}}} \nc{\bp}{{\mathbf{p}}}
\nc{\bq}{{\mathbf{q}}} \nc{\br}{{\mathbf{r}}} \nc{\bt}{{\mathbf{t}}}
\nc{\bfu}{{\mathbf{u}}} \nc{\bv}{{\mathbf{v}}}
\nc{\bx}{{\mathbf{x}}} \nc{\by}{{\mathbf{y}}} \nc{\bz}{{\mathbf{z}}}
\nc{\bw}{{\mathbf{w}}} \nc{\bA}{{\mathbf{A}}}
\nc{\bB}{{\mathbf{B}}} \nc{\bC}{{\mathbf{C}}}
\nc{\bD}{{\mathbf{D}}} \nc{\bF}{{\mathbf{F}}} \nc{\bG}{{\mathbf{G}}}
\nc{\bH}{{\mathbf{H}}} \nc{\bI}{{\mathbf{I}}} \nc{\bJ}{{\mathbf{J}}}
\nc{\bK}{{\mathbf{K}}} \nc{\bM}{{\mathbf{M}}} \nc{\bN}{{\mathbf{N}}}
\nc{\bO}{{\mathbf{O}}} \nc{\bS}{{\mathbf{S}}} \nc{\bT}{{\mathbf{T}}}
\nc{\bU}{{\mathbf{U}}} \nc{\bV}{{\mathbf{V}}} \nc{\bW}{{\mathbf{W}}}
\nc{\bX}{{\mathbf{X}}}
\nc{\bY}{{\mathbf{Y}}} \nc{\bP}{{\mathbf{P}}}
\nc{\bZ}{{\mathbf{Z}}} \nc{\bh}{{\mathbf{h}}}

\nc{\sA}{{\mathsf{A}}} \nc{\sB}{{\mathsf{B}}}
\nc{\sC}{{\mathsf{C}}} \nc{\sD}{{\mathsf{D}}}
\nc{\sE}{{\mathsf{E}}} \nc{\sF}{{\mathsf{F}}} \nc{\sG}{{\mathsf{G}}} \nc{\sH}{{\mathsf{H}}}
\nc{\sI}{{\mathsf{I}}} \nc{\sK}{{\mathsf{K}}} \nc{\sL}{{\mathsf{L}}}
\nc{\sfm}{{\mathsf{m}}} \nc{\sM}{{\mathsf{M}}} \nc{\sN}{{\mathsf{N}}}
\nc{\sO}{{\mathsf{O}}} \nc{\sQ}{{\mathsf{Q}}} \nc{\sP}{{\mathsf{P}}}
\nc{\sT}{{\mathsf{T}}} \nc{\sZ}{{\mathsf{Z}}}
\nc{\sV}{{\mathsf{V}}} \nc{\sW}{{\mathsf{W}}}
\nc{\sfp}{{\mathsf{p}}} \nc{\sq}{{\mathsf{q}}} \nc{\sr}{{\mathsf{r}}}
\nc{\sfs}{{\mathsf{s}}} \nc{\st}{{\mathsf{t}}} \nc{\sfb}{{\mathsf{b}}}
\nc{\sfc}{{\mathsf{c}}} \nc{\sd}{{\mathsf{d}}}
\nc{\sz}{{\mathsf{z}}}

\nc{\tA}{{\widetilde{\mathbf{A}}}}
\nc{\tB}{{\widetilde{\mathcal{B}}}}
\nc{\tg}{{\widetilde{\mathfrak{g}}}} \nc{\tG}{{\widetilde{G}}}
\nc{\TM}{{\widetilde{\mathbb{M}}}{}}
\nc{\tO}{{\widetilde{\mathsf{O}}}{}}
\nc{\tU}{{\widetilde{\mathfrak{U}}}{}} \nc{\TZ}{{\tilde{Z}}}
\nc{\tx}{{\tilde{x}}} \nc{\tbv}{{\tilde{\bv}}}
\nc{\tfP}{{\widetilde{\mathfrak{P}}}{}} \nc{\tz}{{\tilde{\zeta}}}
\nc{\tmu}{{\tilde{\mu}}}

\nc{\urho}{\underline{\rho}} \nc{\uB}{\underline{B}}
\nc{\uC}{{\underline{\mathbb{C}}}} \nc{\ui}{\underline{i}}
\nc{\uj}{\underline{j}} \nc{\ofP}{{\overline{\mathfrak{P}}}}
\nc{\oB}{{\overline{\mathcal{B}}}}
\nc{\og}{{\overline{\mathfrak{g}}}} \nc{\oI}{{\overline{I}}}

\nc{\eps}{\varepsilon} \nc{\hrho}{{\hat{\rho}}} \nc{\balpha}{{\boldsymbol{\alpha}}}
\nc{\blambda}{{\boldsymbol{\lambda}}} \nc{\bmu}{{\boldsymbol{\mu}}} \nc{\bnu}{{\boldsymbol{\nu}}}
\nc{\btheta}{{\boldsymbol{\theta}}} \nc{\bzeta}{{\boldsymbol{\zeta}}} \nc{\bta}{{\boldsymbol{\eta}}}
\nc{\bbeta}{{\boldsymbol{\beta}}} \nc{\bkappa}{{\boldsymbol{\kappa}}} \nc{\bomega}{{\boldsymbol{\omega}}}

\nc{\one}{{\mathbf{1}}} \nc{\two}{{\mathbf{t}}}

\nc\DMO\DeclareMathOperator
\DMO\Sym{Sym}
\nc{\Tot}{{\mathop{\operatorname{\rm Tot}}}}
\nc{\Spec}{\mathop{\operatorname{\rm Spec}}}
\nc{\Ker}{{\mathop{\operatorname{\rm Ker}}}}
\nc{\Isom}{{\mathop{\operatorname{\rm Isom}}}}
\nc{\Hilb}{{\mathop{\operatorname{\rm Hilb}}}}
\nc{\deeq}{{\mathop{\operatorname{\rm deeq}}}}
\nc{\End}{{\mathop{\operatorname{\rm End}}}}
\nc{\Ext}{{\mathop{\operatorname{\rm Ext}}}}
\nc{\Hom}{{\mathop{\operatorname{\rm Hom}}}}
\nc{\CHom}{{\mathop{\operatorname{{\mathcal{H}}\it om}}}}
\nc{\GL}{{\mathop{\operatorname{\rm GL}}}}
\nc{\PGL}{{\mathop{\operatorname{\rm PGL}}}}
\nc{\SL}{{\mathop{\operatorname{\rm SL}}}}
\nc{\SO}{{\mathop{\operatorname{\rm SO}}}}
\nc{\Sp}{{\mathop{\operatorname{\rm Sp}}}}
\nc{\OSp}{{\mathop{\operatorname{\rm SOSp}}}}
\nc{\gr}{{\mathop{\operatorname{\rm gr}}}}
\nc{\Id}{{\mathop{\operatorname{\rm Id}}}}
\nc{\perf}{{\mathop{\operatorname{\rm perf}}}}
\nc{\defi}{{\mathop{\operatorname{\rm def}}}}
\nc{\length}{{\mathop{\operatorname{\rm length}}}}
\nc{\supp}{{\mathop{\operatorname{\rm supp}}}}
\nc{\HC}{{\mathcal H}{\mathcal C}}

\nc{\pr}{{\operatorname{pr}}}
\nc{\Cliff}{{\mathsf{Cliff}}}
\nc{\loc}{{\operatorname{loc}}} \nc{\lc}{{\operatorname{lc}}}
\nc{\Fl}{{\mathbf{Fl}}} \nc{\Ffl}{{\mathcal{F}\ell}}
\nc{\Fib}{{\mathsf{Fib}}}
\nc{\Coh}{{\mathsf{Coh}}} \nc{\FCoh}{{\mathsf{FCoh}}}
\nc{\Perf}{{\mathsf{Perf}}}
\nc{\wtimes}{\mathbin{\widetilde\times}}

\nc{\reg}{{\text{\rm reg}}} \nc{\ren}{{\text{\rm ren}}}
\nc{\self}{{\text{\rm self}}}
\nc{\gvee}{{\mathfrak g}^{\!\scriptscriptstyle\vee}}
\nc{\tvee}{{\mathfrak t}^{\!\scriptscriptstyle\vee}}
\nc{\nvee}{{\mathfrak n}^{\!\scriptscriptstyle\vee}}
\nc{\bvee}{{\mathfrak b}^{\!\scriptscriptstyle\vee}}
       \nc{\rhovee}{\rho^{\!\scriptscriptstyle\vee}}

\nc{\cplus}{{\mathbf{C}_+}} \nc{\cminus}{{\mathbf{C}_-}}
\nc{\cthree}{{\mathbf{C}_*}} \nc{\Qbar}{{\bar{Q}}}

\newcommand\iso{\mathbin{\vphantom{j^{X^2}}\smash{\overset{\sim}{\vphantom{\rule{0pt}{0.20em}}\smash{\longrightarrow}}}}}
\nc{\Gtimes}{\vphantom{j^{X^2}}\smash{\overset{G}{\vphantom{\rule{0pt}{0.30em}}\smash{\times}}}}
\nc{\sGtimes}{\vphantom{j^{X^2}}\smash{\overset{\mathsf G}{\vphantom{\rule{0pt}{0.30em}}\smash{\times}}}}
\newcommand{\oLambda}{\vphantom{j^{X^2}}\smash{\overset{\circ}{\vphantom{\rule{0pt}{0.55em}}\smash{\Lambda}}}}

\nc{\bOmega}{{\overline{\Omega}}}

\nc{\seq}[1]{\stackrel{#1}{\sim}}

\nc{\aff}{{\operatorname{aff}}}
\nc{\fin}{{\operatorname{fin}}}
\nc{\mir}{{\operatorname{mir}}}
\nc{\triv}{{\operatorname{triv}}}
\nc{\ext}{{\operatorname{ext}}}
\nc{\righ}{{\operatorname{right}}}
\nc{\lef}{{\operatorname{left}}}
\nc{\forg}{{\operatorname{forg}}}
\nc{\fid}{{\operatorname{fd}}}
\nc{\odd}{{\operatorname{odd}}}
\nc{\even}{{\operatorname{even}}}
\nc{\modu}{{\operatorname{-mod}}}
\nc{\Gr}{{\operatorname{Gr}}}
\nc{\FT}{{\operatorname{FT}}}
\nc{\Mat}{{\operatorname{Mat}}}
\nc{\MSt}{{\operatorname{MSt}}}
\nc{\sph}{{\operatorname{sph}}}
\nc{\GR}{{\mathbf{Gr}}}
\nc{\Perv}{{\operatorname{Perv}}}
\nc{\Rep}{{\operatorname{Rep}}}
\nc{\Ind}{{\operatorname{Ind}}}
\nc{\IC}{{\operatorname{IC}}}
\nc{\Bun}{{\operatorname{Bun}}}
\nc{\Proj}{{\operatorname{Proj}}}
\nc{\Stab}{{\operatorname{Stab}}}
\nc{\pt}{{\operatorname{pt}}}
\nc{\bfmu}{{\boldsymbol{\mu}}}
\nc{\bfomega}{{\boldsymbol{\omega}}}
\nc{\calM}{\mathcal M}
\nc{\calA}{\mathcal A}
\nc{\calO}{\mathcal O}
\nc{\CC}{\mathcal C}
\nc{\calN}{\mathcal N}
\nc{\grg}{\mathfrak g}
\nc{\dslash}{/\!\!/}
\nc{\tslash}{/\!\!/\!\!/}
\nc\grt{\mathfrak t}
\nc\bfM{\mathbf M}
\nc\bfN{\mathbf N}
\nc\Sig{\Sigma}
\nc\ZZ{\mathbb{Z}}
\nc\calC{\mathcal C}
\nc\calF{\mathcal F}
\nc\calX{\mathcal X}
\nc\calY{\mathcal Y}
\nc\QCoh{\operatorname{QCoh}}
\nc\IndCoh{\operatorname{IndCoh}}
\nc\Maps{\operatorname{Maps}}
\nc\Dmod{D-\operatorname{mod}}
\newcommand\Hecke{\operatorname{Hecke}}
\nc{\calD}{\mathcal D}
\nc\bfO{\mathbf O}

\nc\GG{\mathbb G}
\nc\calK{\mathcal K}
\nc{\calG}{\mathcal G}
\nc\RHom{\operatorname{RHom}}
\nc\Res{\operatorname{Res}}
\nc\Av{\operatorname{Av}}
\nc{\RH}{{\operatorname{RH}}}
\nc{\RT}{{\operatorname{RT}}}
\nc{\DR}{{\operatorname{DR}}}

\nc\grs{\mathfrak s}
\nc{\tilX}{\widetilde X}
\nc\calB{\mathcal B}
\nc\calS{\mathcal S}
\nc\calT{\mathcal T}
\nc\calZ{\mathcal Z}
\nc\LS{\operatorname{LocSys}}
\nc\bfL{\on{\mathbf L}}

\newcommand*\circled[1]
{*{#1}}
\makeatletter
\newcommand{\raisemath}[1]{\mathpalette{\raisem@th{#1}}}
\newcommand{\raisem@th}[3]{\raisebox{#1}{$#2#3$}}
\nc{\binlim}[2][]{\def\@tempa{#1}\@ifnextchar^{\@binlim{#2}}{\@binlim{#2}^{}}}
\def\@binlim#1^#2{\mathbin{\@ifempty{#2}{\mathop{#1}}{\mathop{#1}\@xp\displaylimits\@tempa^{#2}}}}

\makeatother

\nc\cX{{\mathcal X}}

\nc\Gm{{\mathbb G_m}}
\nc{\isoto}
    {\stackrel{\displaystyle\raisemath{-.2ex}{\sim}}\to}
\nc\cD\CalD
\nc\setm\setminus
\nc\cE\CE
\renc\Hecke{\mathit{\CH\kern-.2ex ecke}}
\nc\bsl\backslash
\nc\Fq{\mathbb F_q}
\nc\x\times
\nc\bGO{{\bG_\bO}}
\nc\bla\blambda
\nc\blt\bullet
\nc\opp{{\on{op}}}
\nc\srel\stackrel
\nc\oast\circledast
\nc\tbx{\binlim{\widetilde\boxtimes{}}}
\nc\into\hookrightarrow
\nc\otto\leftrightarrow
\renc\setminus\smallsetminus
\nc\phitau{\varphi\tau}
\newenvironment{i-ii-iii}{%
\begin{enumerate}
}%
{\end{enumerate}}
\nc\ceil[1]{\lceil#1\rceil}  \nc\floor[1]{\lfloor#1\rfloor}
\nc\Lie{\on{Lie}}
\nc\sS{{\mathsf S}}
\nc\vvv{\ensuremath{\red\surd}}
\nc\la\lambda
\def\arxiv#1{\href{http://arxiv.org/abs/#1}{\tt arXiv:#1}} \let\arXiv\arxiv
\nc\kap{\kappa}
\nc\gra{\mathfrak a}
\nc\gl{\mathfrak{gl}}
\nc\sTr{\operatorname{sTr}}
\nc\hatG{\widehat{G}}
\nc\calL{\mathcal L}
\nc\Whit{\operatorname{Whit}}
\nc\KL{\operatorname{KL}}

\makeatletter
\renewcommand{\subsection}{\@startsection{subsection}{2}{0pt}{-3ex
plus -1ex minus -0.2ex}{-2mm plus -0pt minus
-2pt}{\normalfont\bfseries}} \makeatother

\numberwithin{equation}{subsection}
\allowdisplaybreaks
\nc\mto{\mapsto }
\nc\en{\enspace }

\begin{document}

\author[A.Braverman]{Alexander Braverman}
\address{Department of Mathematics, University of Toronto and Perimeter Institute
of Theoretical Physics, Waterloo, Ontario, Canada, N2L 2Y5}
\email{braval@math.toronto.edu}

\author[M.Finkelberg]{Michael Finkelberg}
\address{Einstein Institute of Mathematics, The Hebrew University of Jerusalem,
  Edmond J. Safra Campus, Giv’at Ram, Jerusalem, 91904, Israel;
\newline  National Research University Higher School of Economics}
\email{fnklberg@gmail.com}


\author[R.Travkin]{Roman Travkin}
\address{Skolkovo Institute of Science and Technology, Moscow, Russia}
\email{roman.travkin2012@gmail.com}

\title
{Orthosymplectic Satake equivalence, II}
\dedicatory{To Mikhail Kapranov on his 60th birthday}




\begin{abstract}
  This is a companion paper of~\cite{bfgt,bft}.
  We prove an equivalence relating representations of a degenerate orthosymplectic supergroup
  with the category of twisted $\Sp(2n,\BC[\![t]\!])$-equivariant $D$-modules on the so called mirabolic
  affine Grassmannian of $\Sp(2n)$. We also discuss (conjectural) extension of this equivalence to
  the case of quantum supergroups and to some exceptional supergroups.
\end{abstract}

\maketitle

\tableofcontents

\section{Introduction}
\label{intro}

\subsection{Orthosymplectic Satake equivalence}
This paper is a sequel to~\cite{bft} where a conjectural version of the Satake equivalence for
orthosymplectic groups $\on{SOSp}(2k|2n)$ was studied. The main goal of the present paper is
to extend this study to the case of $\on{SOSp}(2k+1|2n)$. The `central' case is $k=n$, and this
is the only case where we prove an equivalence of categories proposed by D.~Gaiotto.

Namely, we construct a geometric realization of the category $\Rep(\ul{\on{SOSp}}(2n+1|2n))$
of finite-dimensional representations of the degenerate supergroup $\ul{\on{SOSp}}(2n+1|2n)$.
Recall that the even part of $\ul{\on{SOSp}}(2n+1|2n)$ is $\SO(2n+1)\times\Sp(2n)$, and its Lie
superalgebra $\ul{\osp}(2n+1|2n)$ has trivial supercommutator of two odd elements, while the
supercommutator of an even element with another element is the same as in $\osp(2n+1|2n)$.

This geometric realization makes use of the `mirabolic' affine Grassmannian $V_\bF\times\Gr_G$
of the symplectic group $G=\Sp(V)$, where $V$ is a $2n$-dimensional symplectic complex vector space,
and $\bF=\BC(\!(t)\!)\supset\BC[\![t]\!]=\bO$.
The symplectic form on $V$ induces a (complex-valued) symplectic form on the Tate vector space
$V_\bF$, and it gives rise to the Weyl algebra $\CW$ of $V_\bF$. The symplectic mirabolic category
is the category of $G_\bO$-equivariant objects in the tensor product of $D$-modules on $\Gr_G$
and $\CW$-modules. The category $\CW\modu$ is equipped with a {\em twisted} action of
$\mathrm{D}\modu_{-1/2}(G_\bF)$ at the level $-1/2$. Hence there are two equivalent incarnations
of the symplectic mirabolic category. First, $\CalD\CW\modu^{G_\bO}$ are the $G_\bO$-equivariant
objects of $\mathrm{D}\modu_{1/2}(\Gr_G)\otimes\CW\modu$ with respect to the {\em untwisted}
strong action of $G_\bF$. Second, $\CalD'\CW\modu^{G'_\bO}$ are the $G_\bO$-equivariant
objects of $\mathrm{D}\modu(\Gr_G)\otimes\CW\modu$ with respect to the $-1/2$-{\em twisted}
strong action of $G_\bF$.

Our construction of the equivalence $\Rep(\ul{\on{SOSp}}(2n+1|2n))\iso\CalD\CW\modu^{G_\bO}$
follows the pattern of~\cite{bft}. Namely, we first construct an equivalence from
$D^{\SO(V_0)\times\Sp(V)}_{\on{perf}}(\Sym^\bullet(\Pi(V_0\otimes V)[-1]))$ to a derived version of
$\CalD\CW\modu^{G_\bO}$. Here $V_0$ is a $2n+1$-dimensional complex vector space equipped with
a nondegenerate symmetric bilinear form. Then we precompose this equivalence with the Koszul
equivalence $D^b\Rep(\ul{\on{SOSp}}(2n+1|2n))\iso
D^{\SO(V_0)\times\Sp(V)}_{\on{perf}}(\Sym^\bullet(\Pi(V_0\otimes V)[-1]))$. Finally, we check that the
composed equivalence respects the natural $t$-structures.

Compared to~\cite{bft} we have to overcome two new difficulties. First, $\CW\modu$ is equivalent
to the category of $D$-modules on $V_\bF/V_\bO$, but the set of $G_\bO$-orbits in
$(V_\bF/V_\bO)\times\Gr_G$ is uncountable, and we have to classify the discrete family of
{\em relevant} orbits which do carry $G_\bO$-equivariant $D$-modules. It turns out that this
family is naturally indexed by the dominant weights of $\SO(V_0)\times\Sp(V)$.
Second, the computation of equivariant Ext's in~\cite{bft} was reduced to the calculations with
equivariant cohomology, but in the present case, due to the twisting, the equivariant (De Rham)
cohomology simply vanishes, and we have to resort to the endoscopic arguments of~\cite{dlyz}.

Otherwise, the arguments are quite parallel to those of~\cite{bft}, and they are just briefly
indicated after introducing an appropriate setup.

\subsection{Further generalizations: general orthosymplectic case, quantum supergroups and
  exceptional supergroups}
The above results only involve the category of representations of the degenerate supergroup
$\ul{\on{SOSp}}(2n+1|2n)$.
Following D.~Gaiotto, we formulate in~\S\ref{osp2n} similar conjectures relating
$\Rep(\ul{\on{SOSp}}(2k+1|2n))$ and $\Rep(\ul{\on{SOSp}}(2n+1|2k))$ (for $k<n$) with certain
equivariant objects of $D\modu(\Gr_G)\otimes\CW(\bF^{2k})\modu$ (the Weyl algebra of a symplectic
Tate space $\bF^{2k}$). The equivariance is taken with respect to the semidirect product
$\Sp(2k,\bO)\ltimes\on{U}_k(\bF)$ for a certain unipotent subgroup $\on{U}_k\subset\Sp(2n)$
related to a `hook' nilpotent $e\in\fsp(2n)$ of Jordan type $(2n-2k,1^{2k})$.

We also formulate a quantum version of these conjectures for representations of quantum supergroups
$\Rep_q(\on{SOSp}(2k+1|2n))$ and $\Rep_q(\on{SOSp}(2n+1|2k))$. In~\S\ref{glkn} we reformulate the
Gaiotto conjectures~\cite[\S2.6]{bfgt} for representations of quantum supergroups $\Rep_q(\GL(K|N))$
in a similar vein. Namely, we replace a Whittaker-type equivariance condition with respect to
a certain unipotent group by a plain equivariance condition (no character) with respect to a
larger unipotent group, at the expense of adding a factor of $\CW(\bF^{2K})\modu$, cf.~\cite[\S5.3]{ty}.

Finally, in~\S\S\ref{f4},\ref{g3} we formulate similar conjectures for representations of the
exceptional quantum supergroups $\Rep_q(\on{F}(4))$ and $\Rep_q(\on{G}(3))$.

We refer the reader to \cite[\S2]{bfgt} where the meaning of very similar statements (for the
supergroup $\GL(M|N)$) is explained from the point of view of local geometric Langlands correspondence.

\subsection{Acknowledgments}
We are grateful to G.~Dhillon, A.~Elashvili, D.~Gaiotto, A.~Hanany, M.~Jibladze, D.~Leites, I.~Motorin,
H.~Nakajima, V.~Ostrik, S.~Raskin, V.~Serganova, D.~Timashev, A.~Tsymbaliuk, I.~Ukraintsev and R.~Yang
for very useful discussions. We are also obliged to the anonymous referee for his very careful reading
of our manuscript, spotting a few mistakes and suggesting a few improvements.

A.B.\ was partially supported by NSERC. The research of M.F.~was supported by the Israel Science Foundation
(grant No.~994/24).

\section{A coherent realization of $\CalD\CW\modu^{\Sp(2n,\bO),\lc}$}

\subsection{Symplectic mirabolic category}
\subsubsection{Weyl algebra}
\label{weyl}
We fix a $2n$-dimensional complex vector space $V$ with symplectic form $\langle\, ,\rangle$.
Let $G=\Sp(2n)=\Sp(V)$, and $\fg=\fsp(2n)=\fsp(V)$. The symplectic form on $V$ extends to the same
named $\BC$-valued symplectic form on
$\bV=V_\bF\colon \langle f,g\rangle=\on{Res}\,\langle f,g\rangle_\bF dt$. We denote by $\CW$ the
completion of the Weyl algebra of $(\bV,\langle\, ,\rangle)$ with respect to the left ideals
generated by the compact open lattices of $\bV$.

We consider the dg-category $\CW\modu$ of discrete $\CW$-modules. More concretely, we identify
$\CW$ with the ring of differential operators on a Lagrangian discrete lattice $L\subset\bV$,
e.g.\ $L=t^{-1}V_{\BC[t^{-1}]}$. Then $\CW\modu$ is the inverse limit of $\mathrm{D}\modu(U)$ over
finite dimensional subspaces $U\subset L$ with respect to the functors $i^!_{U\hookrightarrow U'}$.
Equivalently, $\CW\modu$ is the colimit of $\mathrm{D}\modu(U)$ with respect to the functors
$i_{U\hookrightarrow U',*}$.

There is a twisted action $\mathrm{D}\modu_{-1/2}(G_\bF)\circlearrowright\CW\modu$~\cite[\S10]{r}.
Here~$-1/2$ stands for the~$-1/2$-multiple of the level of $G_\bF$ corresponding to the trace
form of $V$ on $\fg$.

\subsubsection{Satake equivalence}
\label{satake}
We consider the dg-category $\mathrm{D}\modu_{1/2}^{G_\bO,\lc}(\Gr_G)\subset\mathrm{D}\modu_{1/2}^{G_\bO}(\Gr_G)$ of
{\em locally compact} $G_\bO$-equivariant $D$-modules on the affine Grassmannian $\Gr_G$
twisted by the square root of the very ample determinant line bundle $\CalD$. Here `locally compact'
means compact when regarded as a plain $D$-module on $\Gr_G$. The {\em twisted Satake equivalence}
(the composition of~\cite[Example 1.10]{dlyz} with the derived Satake equivalence of~\cite{bf} for
$\SO(2n+1)$) is a monoidal equivalence
\[\bbeta^\fg\colon D^G_\perf(\Sym^\bullet(\fg[-2]))\iso\mathrm{D}\modu_{1/2}^{G_\bO,\lc}(\Gr_G)\]
from the dg-category of $G$-equivariant perfect dg-modules over the dg-algebra $\Sym^\bullet(\fg[-2])$
equipped with trivial differential.
It extends to the same named monoidal equivalence of the Ind-completions
\[\bbeta^\fg\colon D^G(\Sym^\bullet(\fg[-2]))\iso\mathrm{D}\modu_{1/2}^{G_\bO,\ren}(\Gr_G).\]
Here $\mathrm{D}\modu_{1/2}^{G_\bO,\ren}(\Gr_G)$ stands for the {\em renormalized}
category~\cite[\S12.2.3]{aga}.
It also extends to the same named monoidal equivalence
\begin{equation}
  \label{nilp twist}
  \bbeta^\fg\colon D^G_{\on{nilp}}(\Sym^\bullet(\fg[-2]))\iso\mathrm{D}\modu_{1/2}^{G_\bO}(\Gr_G),
\end{equation}
similarly to~\cite[Corollary 12.5.5]{aga}. Here $D^G_{\on{nilp}}(\Sym^\bullet(\fg[-2]))$ stands for
the category of $G$-equivariant dg-modules over $\Sym^\bullet(\fg[-2])$ with nilpotent support.

We will also use the untwisted Satake equivalence of~\cite{bf}
\[\bbeta^{\gvee}\colon D^{G^\vee}_\perf(\Sym^\bullet(\gvee[-2]))\iso\mathrm{D}\modu^{G_\bO,\lc}(\Gr_G),\]
along with its Ind-completion
\[\bbeta^{\gvee}\colon D^{G^\vee}(\Sym^\bullet(\gvee[-2]))\iso\mathrm{D}\modu^{G_\bO,\ren}(\Gr_G).\]
Here $G^\vee=\SO(2n+1)=\SO(V_0)$ is the Langlands dual Lie group, while $\gvee=\fso(2n+1)=\fso(V_0)$
is the Langlands dual Lie algebra, and the $2n+1$-dimensional vector space $V_0$ is equipped with a
nondegenerate symmetric bilinear form $(\, ,)$.

It also extends to the same named monoidal equivalence~\cite[Corollary 12.5.5]{aga}
\begin{equation}
  \label{nilp}
  \bbeta^{\gvee}\colon D^{G^\vee}_{\on{nilp}}(\Sym^\bullet(\gvee[-2]))\iso\mathrm{D}\modu^{G_\bO}(\Gr_G).
\end{equation}

Finally, for $U_0\in\Rep(G^\vee)$ (resp.\ $U_1\in\Rep(G)$) we will denote
$\bbeta^{\gvee}(U_0\otimes\Sym^\bullet(\gvee[-2]))$ (resp.\ $\bbeta^\fg(U_1\otimes\Sym^\bullet(\fg[-2]))$)
by ${\mathbf S}^{\gvee}(U_0)$ (resp.\ ${\mathbf S}^\fg(U_1)$).

\subsubsection{Symplectic mirabolic category}
\label{mirab}
We denote by $\CalD\CW\modu$ the tensor product $\mathrm{D}\modu_{1/2}(\Gr_G)\otimes\CW\modu$.
We consider the category $\CalD\CW\modu^{G_\bO,\lc}$ of locally compact $G_\bO$-equivariant objects and
its Ind-completion $\CalD\CW\modu^{G_\bO,\ren}$.

The {\em right} convolution action of $\mathrm{D}\modu_{1/2}^{G_\bO,\lc}(\Gr_G)$ on itself gives rise to
the action of $\mathrm{D}\modu_{1/2}^{G_\bO,\lc}(\Gr_G)$ on $\CalD\CW\modu^{G_\bO,\lc}$.
Furthermore, recall the twisted action $\mathrm{D}\modu_{-1/2}(G_\bF)\circlearrowright\CW\modu$.
The action of $G_\bF$ on $\Gr_G$ gives rise to the action
$\mathrm{D}\modu_{1/2}(G_\bF)\circlearrowright\mathrm{D}\modu_{1/2}(\Gr_G)$.
These two actions combine to give the diagonal {\em non-twisted} action
$\mathrm{D}\modu(G_\bF)\circlearrowright\CalD\CW\modu$. Finally, this latter action gives rise
to the left action $\mathrm{D}\modu^{G_\bO,\lc}(\Gr_G)\circlearrowright\CalD\CW\modu^{G_\bO,\lc}$.

All in all, we obtain an action \[\mathrm{D}\modu^{G_\bO,\lc}(\Gr_G)\otimes
\mathrm{D}\modu_{1/2}^{G_\bO,\lc}(\Gr_G)\circlearrowright\CalD\CW\modu^{G_\bO,\lc}.\]

We will also need another realization $\CalD'\CW\modu^{G'_\bO,\lc}\cong\CalD\CW\modu^{G_\bO,\lc}$:
we have the natural equivalences
\begin{multline}
  \label{two real}
  \CalD\CW\modu^{G_\bO}=(\mathrm{D}\modu_{1/2}(\Gr_G)\otimes\CW\modu)^{G_\bO}\\
  \cong(\mathrm{D}\modu(\Gr_G)\otimes\mathrm{D}\modu_{1/2}(\Gr_G)\otimes\CW\modu)^{G_\bF}\\
  \cong(\mathrm{D}\modu(\Gr_G)\otimes\CW\modu)^{G'_\bO}=:\CalD'\CW\modu^{G'_\bO}.
\end{multline}
Here the action of $G_\bF$ on $\mathrm{D}\modu(\Gr_G)\otimes\CW\modu$ is {\em twisted}:
$\mathrm{D}\modu_{-1/2}(G_\bF)\circlearrowright\mathrm{D}\modu(\Gr_G)\otimes\CW\modu$. So
$(\mathrm{D}\modu(\Gr_G)\otimes\CW\modu)^{G'_\bO}$ denotes the category of $G_\bO$-equivariant
objects with respect to this twisted action of $G_\bO\subset G_\bF$.
It is equipped with the right action of $\mathrm{D}\modu^{G_\bO}(\Gr_G)$ and the left action of
$\mathrm{D}\modu_{-1/2}^{G_\bO}(\Gr_G)$.

The equivalence $\CalD\CW\modu^{G_\bO}\cong\CalD'\CW\modu^{G'_\bO}$ is compatible with the actions of
$\mathrm{D}\modu^{G_\bO}(\Gr_G)$ and $\mathrm{D}\modu_{\pm1/2}^{G_\bO}(\Gr_G)$ in the following way.
First of all, the inversion $g\mapsto g^{-1}\colon G_\bF\to G_\bF$ gives rise to a monoidal
anti-involution of $\mathrm{D}\modu^{G_\bO}(\Gr_G)$. The left action of $\mathrm{D}\modu^{G_\bO}(\Gr_G)$
on $\CalD\CW\modu^{G_\bO}$ goes to the right action of $\mathrm{D}\modu^{G_\bO}(\Gr_G)$ on
$\CalD'\CW\modu^{G'_\bO}$ composed with the above anti-involution. Let $\fC$ denote the Chevalley
involution of $G$ (the canonical outer automorphism of $G$ interchanging conjugacy classes of
$g$ and $g^{-1}$). We keep the same notation for the induced involution of
$\mathrm{D}\modu^{G_\bO}(\Gr_G)$. Then the above anti-involution coincides with $\fC$ (due to the
commutativity constraint on $\mathrm{D}\modu^{G_\bO}(\Gr_G)$, there is no difference between
involutions and anti-involutions). Finally,
notice that $\fC$ is trivial since the Dynkin graph of $G$ has no automorphisms.
Second, in a similar vein, the right action of $\mathrm{D}\modu_{1/2}^{G_\bO}(\Gr_G)$ on
$\CalD\CW\modu^{G_\bO}$ goes to the left action of $\mathrm{D}\modu_{-1/2}^{G_\bO}(\Gr_G)$ on
$\CalD'\CW\modu^{G'_\bO}$ composed with the twisting by the inverse determinant line bundle
$\CalD^{-1}$ on $\Gr_G$.

\subsubsection{Theta-sheaf}
\label{theta}
Let $\GR_G$ be the Kashiwara (infinite type) scheme version of the affine Grassmannian of $G$:
it is the moduli space of $G$-bundles on $\BP^1$ equipped with a trivialization in the formal
neighbourhood of $\infty\in\BP^1$. Recall the Radon Transform
\[\RT_!\colon\mathrm{D}\modu_{-1/2}^{G_\bO}(\Gr_G)\iso\mathrm{D}\modu_{-1/2}^{G_\bO}(\GR_G)_!\]
(see e.g.~\cite[B.7]{bdfrt}). We will use the inverse equivalence $\RT_!^{-1}$.

We will also use the $G_\bO$-equivariant $-1/2$-twisted $D$-module
$\varTheta$ on $\GR_G$, see~\cite[\S2.4]{bdfrt}. It was introduced in~\cite[\S2]{laf}.
The perverse sheaf corresponding
to $\varTheta$ under the Riemann-Hilbert correspondence was introduced in~\cite{lys} and studied
in~\cite{ll} (over the base field $\BF_q$). We will use the following key relation between the
theta-sheaf and the twisted Satake equivalence. First we introduce the notation

\begin{equation}
  \label{R}
  \CR:=\RT_!^{-1}(\varTheta).
\end{equation}

We choose a pair of opposite maximal unipotent subgroups $U_G,U_G^-\subset G$, their regular
characters $\psi,\psi^-$, and denote by
$\bkappa^\fg\colon D^G(\Sym^\bullet(\fg[-2]))\to D(\BC[\Xi_\fg])$ the functor of Kostant-Whittaker
reduction with respect to $(U_G^-,\psi^-)$
(see e.g.~\cite[\S2]{bf}). Here $\Xi_\fg$ with grading disregarded is the tangent bundle $T\Sigma_\fg$
of the Kostant slice $\Sigma_\fg\subset\fg^*$. Let us write $\kappa$ for the
Ad-invariant bilinear form on $\fg$, i.e.\ level, corresponding to the central
charge of $-1/2$. In other words, we have $\kappa=\kappa_c-\frac{2n+1}{2n+2}\kappa_c$
(where $\kappa_c$ is the critical level) since the dual Coxeter number of $\fsp(2n)$ is $n+1$.

If we consider the Langlands dual Lie algebra $\gvee\simeq\fso_{2n+1}$, the
form $\kappa$  gives rise to identifications $\Sigma_\fg\cong\Sigma_{\gvee}$ and
$\Xi_\fg\cong\Xi_{\gvee}$. Also, we have a canonical isomorphism
$H^\bullet_{G_\CO}(\Gr_G)\cong\BC[\Xi_{\gvee}]\cong\BC[\Xi_\fg]$. This is a theorem of
V.~Ginzburg~\cite{g1} (for a published account see e.g.~\cite[Theorem 1]{bf}).

We also have the Kostant-Whittaker reduction functor (see~\cite[\S2]{bf})
\[\bkappa^{\gvee}\colon D^{G^\vee}(\Sym^\bullet(\gvee[-2]))\to D(\BC[\Xi_{\gvee}])\cong D(\BC[\Xi_\fg]).\]
By~\cite[Theorem 4]{bf}, we have a canonical isomorphism
\begin{equation}
  \label{glob coh untwisted}
  H^\bullet_{G_\bO,\DR}(\Gr_G,\bbeta^{\gvee}(M))\cong\bkappa^{\gvee}M
\end{equation}
of $H^\bullet_{G_\CO}(\Gr_G)\cong\BC[\Xi_\fg]$-modules for any $M\in D^{G^\vee}(\Sym^\bullet(\gvee[-2]))$.

Now given $\CF\in\mathrm{D}\modu_{1/2}^{G_\bO,\ren}(\Gr_G)$ we consider the tensor product
$\CF\overset{!}{\otimes}\CR$.
Since the twistings of the factors cancel out, the tensor product is an {\em untwisted}
$G_\bO$-equivariant $D$-module, and we can consider its equivariant De Rham cohomology.
The aforementioned key relation is a canonical isomorphism
\begin{equation}
  \label{glob coh}
  H^\bullet_{G_\bO,\DR}(\Gr_G,\bbeta^\fg(M) \overset !\otimes\CR)\cong\bkappa^\fg M
\end{equation}
of $H^\bullet_{G_\CO}(\Gr_G)\cong\BC[\Xi_\fg]$-modules for any $M\in D^G(\Sym^\bullet(\fg[-2]))$.
Indeed, according to~\cite[Example 1.10]{dlyz}, we have a monoidal equivalence of categories
$\boldsymbol{\varepsilon}\colon\mathrm{D}\modu_{1/2}^{G_\bO}(\Gr_G)\iso\mathrm{D}\modu^{G^\vee_\bO}(\Gr_{G^\vee})$.
It follows from~\cite[Proposition 9.54]{dlyz} and~\cite[Proposition 4]{bf} that this equivalence takes
$\CR\otimes\CalD$ to the dualizing sheaf $\bomega_{\Gr_{G^\vee}}\in\mathrm{D}\modu^{G^\vee_\bO}(\Gr_{G^\vee})$.
Furthermore, \[H^\bullet_{G_\bO,\DR}(\Gr_G,\bbeta^\fg(M) \overset !\otimes\CR)\cong\Hom({\mathbf 1}_G,
\bbeta^\fg(M)*(\CR\otimes\CalD)),\] where ${\mathbf 1}_G\in \mathrm{D}\modu_{1/2}^{G_\bO}(\Gr_G)$ is the
monoidal unit, and $*$ stands for the convolution operation. 
Finally, \[H^\bullet_{G^\vee_\bO,\DR}(\Gr_{G^\vee},
\boldsymbol{\varepsilon}\bbeta^\fg(M))\cong\Hom({\mathbf 1}_{G^\vee},
\boldsymbol{\varepsilon}\bbeta^\fg(M)*\bomega_{\Gr_{G^\vee}}),\]
where ${\mathbf 1}_{G^\vee}\in \mathrm{D}\modu^{G^\vee_\bO}(\Gr_{G^\vee})$ is the
monoidal unit, and $*$ stands for the convolution operation.

\subsection{$D$-module realization of $\CalD\CW\modu^{G_\bO,\lc}$}
\label{constr}
Let $\on{Heis}$ stand for the Heisenberg central extension of $V_\bF$ by $\BG_a$ (canonically split
after restriction to $V_\bO$). Let $\chi$ be the character $D$-module on $G_\bO\ltimes V_\bO\times\BG_a$
equal to the pullback of the exponential $D$-module on $\BG_a$ with respect to the projection
$G_\bO\ltimes V_\bO\times\BG_a\to\BG_a$. We consider the category
$\mathrm{D}\modu_{1/2}(\Gr_G\times\on{Heis}):=
\mathrm{D}\modu_{1/2}(\Gr_G)\otimes\mathrm{D}\modu(\on{Heis})$.
It is equipped with the
strong diagonal action of $G_\bO\ltimes V_\bO\times\BG_a$ (the action on $\mathrm{D}\modu_{1/2}(\Gr_G)$
factors through the quotient $G_\bO$). We consider the category
$\mathrm{D}\modu_{1/2}^{G_\bO\ltimes V_\bO\times\BG_a,\chi,\lc}(\Gr_G\times\on{Heis})$ of locally compact
$\chi$-equivariant objects.

\begin{lem}
  \label{heis}
  There is an equivalence of categories
  \[\mathrm{D}\modu_{1/2}^{G_\bO\ltimes V_\bO\times\BG_a,\chi,\lc}(\Gr_G\times\on{Heis})\iso
  \CalD\CW\modu^{G_\bO,\lc}.\]
\end{lem}

\begin{proof}
  Let $\IC_0\in\mathrm{D}\modu_{1/2}(\Gr_G)^{G_\bO,\lc,\heartsuit}$ denote the irreducible twisted $D$-module
  supported at the base point of $\Gr_G$. Then the tensor product
  $\IC_0\otimes\BC[V_\bO]\in\CalD\CW\modu$ is strongly $(G_\bO\ltimes V_\bO\times\BG_a,\chi)$-equivariant,
  and so gives rise to a functor from
  $\mathrm{D}\modu_{1/2}^{G_\bO\ltimes V_\bO\times\BG_a,\chi,\lc}(\Gr_G\times\on{Heis})$ to
  $\CalD\CW\modu^{G_\bO,\lc}$ that is the desired equivalence.
\end{proof}

\begin{rem}
  A similar category of $\ell$-adic sheaves over the base field $\BF_q$ was studied in~\cite{ll}.
\end{rem}

\begin{rem}
  \label{prime}
Another incarnation $\CalD'\CW\modu^{G'_\bO,\lc}$ of the mirabolic category $\CalD\CW\modu^{G_\bO,\lc}$
(see~\S\ref{mirab}) has a similar $D$-module realization. We consider the category
$\mathrm{D}\modu(\Gr_G\times\on{Heis})=
\mathrm{D}\modu(\Gr_G)\otimes\mathrm{D}\modu(\on{Heis})$. It is equipped with the
strong diagonal action of $G_\bO\ltimes V_\bO\times\BG_a$ (the action on $\mathrm{D}\modu(\Gr_G)$
factors through the quotient $G_\bO$). We consider the category
$\mathrm{D}\modu^{G_\bO\ltimes V_\bO\times\BG_a,\chi,\lc}(\Gr_G\times\on{Heis})$ of locally compact
$\chi$-equivariant objects. Then we have an equivalence
$\mathrm{D}\modu^{G_\bO\ltimes V_\bO\times\BG_a,\chi,\lc}(\Gr_G\times\on{Heis})\iso\CalD'\CW\modu^{G'_\bO,\lc}$.
\end{rem}

\subsection{Fusion}
Let $X$ be a smooth curve. For any integer $k>0$, and a collection $x=(x_i)_{i=1}^k$ of
$S$-points of $X$, we denote by $\CalD_x$ the formal neighborhood of the union of graphs
$|x| := \bigcup_{i=1}^k \Gamma_{x_i}\subset S\times X$, and we denote by
$\cD_x^\circ :=\CalD_x \setminus |x|$ the punctured formal neighborhood.

The mirabolic version of the Beilinson-Drinfeld Grassmannian is the
     ind-scheme $\Gr^\mir_{G,BD,k}$ over $X^k$ parametrizing the following collections of data:
\[ (x_i)_{i=1}^k,\ \CE,\ \phi\colon \CE_\triv|_{\CalD_x^\circ}\iso\CE|_{\CalD_x^\circ},\
v\in \Gamma(\CalD_x^\circ,\cE) , a\in\BG_a\]
where $\CE$ is a rank $2n$ vector bundle on $\CalD_x$ equipped with a symplectic form.
In case $X=\BA^1$, over the complement to the diagonals we have a canonical isomorphism
\[(\BA^k\setminus\Delta)\x_{\BA^k}\Gr^\mir_{G,BD,k}\cong
(\BA^k\setminus\Delta)\times(\Gr_G\times V_\bF)^k\times\BG_a.\] We denote the projection
$(\BA^k\setminus\Delta)\times(\Gr_G\times V_\bF)^k\times\BG_a\to(\Gr_G\times V_\bF)^k\times\BG_a$
by $\on{pr}_2$.
Given $\CF_1,\CF_2\in\mathrm{D}\modu_{1/2}^{G_\bO\ltimes V_\bO\times\BG_a,\chi,\lc}(\Gr_G\times\on{Heis})$,
we take $k=2$ and define the fusion
\[\CF_1\star\CF_2:=\on{p}_{2*}\psi_{x-y}\on{a}_*\on{pr}_2^*(\CF_1\boxtimes\CF_2),\]
where $x,y$ are coordinates on $\BA^2$ (so that $x-y=0$ is the equation of the diagonal
$\Delta\subset\BA^2$), and $\psi_{x-y}$ is the nearby cycles functor for the pullback of the
function $x-y$ to $\Gr^\mir_{G,BD,2}$. Also, $\on{p}_2$ in the above definition projects
$\BA^1\times\Gr_G\times V_\bF\times\BG_a=\BA^1\times\Gr_G\times\on{Heis}$ to $\Gr_G\times\on{Heis}$,
while $\pr_2$ projects $(\BA^2\setminus\Delta)\times(\Gr_G\times\on{Heis})^2$ to $(\Gr_G\times\on{Heis})^2$.
Finally, $\on{a}\colon (\BA^2\setminus\Delta)\times(\Gr_G\times\on{Heis})^2
=(\BA^2\setminus\Delta)\times(\Gr_G\times V_\bF)^2\times\BG_a^2\to(\BA^2\setminus\Delta)\times(\Gr_G\times V_\bF)^2\times\BG_a$
is the addition morphism along the third factor: $\BG_a^2\to\BG_a$.

\subsection{Orthogonal and symplectic Lie algebras}
\label{sosp}
We adapt~\cite[\S2.1]{bft} to our present setup. The tensor product $V_0\otimes V$ is equipped with a
nondegenerate skew-symmetric bilinear form $(\, ,)\otimes\langle\, ,\rangle$. It is preserved by the
action of the group $\SO(V_0)\times\Sp(V)$.

Our nondegenerate bilinear forms on $V_0,V$ define identifications
$V_0\cong V_0^*,\ V\cong V^*$. In particular, $V_0\otimes V$ is identified with
$V_0^*\otimes V=\Hom(V_0,V)$. Given $A\in\Hom(V_0,V)$ we have the adjoint operator
$A^t\in\Hom(V,V_0)$. We have the moment maps
\[\bq_0\colon V_0\otimes V\to\fso(V_0)^*,\ A\mapsto A^tA,\ \operatorname{and}\
\bq\colon V_0\otimes V\to\fsp(V)^*,\ A\mapsto AA^t,\]
where we make use of the identification $\fso(V_0)\cong\fso(V_0)^*$
(resp.\ $\fsp(V)\cong\fsp(V)^*$)
via the trace form (resp.\ {\em negative} trace form) of the defining representation.
Note also that the complete moment map $(\bq_0,\bq)$ coincides with the ``square''
(half-self-supercommutator) map on the odd part $\bg_{\bar1}$ of the orthosymplectic Lie
superalgebra $\bg$. We define the odd nilpotent cone $\CN_{\bar1}\subset V_0\otimes V$
as the {\em reduced} subscheme cut out by
the condition of nilpotency of $A^tA$ (equivalently, by the condition of nilpotency of $AA^t$).

We choose Cartan subalgebras $\ft_0\subset\fso(V_0)$ and $\ft\subset\fsp(V)$.
We choose a basis $\varepsilon_1,\ldots,\varepsilon_n$ in $\ft_0^*$ such that the Weyl group
$W_0=W(\fso(V_0),\ft_0)$ acts by permutations of basis elements and by the sign changes
of basis elements, and the roots of $\fso(V_0)$ are given by $\{\pm\eps_i\pm\eps_j,\ i\neq j;\ \pm\eps_i\}$.
We set $\Sigma_{\gvee}=\ft_0^*\dslash W_0$.
We also choose a basis $\delta_1,\ldots,\delta_n$ in $\ft^*$ such that the Weyl group
$W=W(\fsp(V),\ft)$ acts by permutations of basis elements and by the sign changes
of basis elements, and the roots of $\fsp(V)$ are given by $\{\pm\delta_i\pm\delta_j,\ i\neq j;\
\pm2\delta_i\}$. We set $\Sigma_\fg=\ft^*\dslash W$.

We identify $\ft_0^*\cong\ft^*,\ \varepsilon_i\mapsto\delta_i$,
and this identification gives rise to an isomorphism $\varPi\colon \Sigma_{\gvee}\to\Sigma_\fg$.

Recall (see e.g.~\cite[\S\S2.1,2.6]{bf}) that $\Sigma_{\gvee}$ is embedded as a Kostant slice
into the open set of regular elements $(\fso(V_0)^*)^\reg\subset\fso(V_0)^*$, and $\Sigma_\fg$
is embedded into $(\fsp(V)^*)^\reg$. Furthermore, these slices $\Sigma_{\gvee},\Sigma_\fg$ carry the
universal centralizer sheaves of abelian Lie algebras $\fz_{\gvee},\fz_\fg$.
Given an $\SO(V_0)$-module $M$ (resp.\ an $\Sp(V)$-module $M'$), we have the corresponding graded
$\Gamma(\Sigma_{\gvee},\fz_{\gvee})$-module $\bkappa^{\gvee}(M\otimes\Sym^\bullet(\gvee[-2]))$ (resp.\ the
$\Gamma(\Sigma_\fg,\fz_\fg)$-module $\bkappa^\fg(M'\otimes\Sym^\bullet(\fg[-2]))$)
(the {\em Kostant functor} of {\em loc.cit.}, cf.\ notation of~\S\ref{theta}).
Since the universal enveloping algebra $U(\fz_{\gvee})$ (resp.\ $U(\fz_\fg)$) is identified in
{\em loc.cit.}\ with the sheaf of functions on the tangent bundle $T\Sigma_{\gvee}\cong\Xi_{\gvee}$
(resp.\ $T\Sigma_\fg\cong\Xi_\fg$), we will use the same notation
$\bkappa^{\gvee}(M\otimes\Sym^\bullet(\gvee[-2])),\bkappa^\fg(M'\otimes\Sym^\bullet(\fg[-2]))$
for the corresponding coherent sheaves on $\Xi_{\gvee}, \Xi_\fg$.
Finally, according to the previous paragraph, we have the isomorphism
$d\varPi\colon \Xi_{\gvee}=T\Sigma_{\gvee}\to T\Sigma_\fg=\Xi_\fg$.

We choose Borel subalgebras $\ft_0\subset\fb_0\subset\fso(V_0)$ corresponding to the
choice of positive roots $R_0^+=\{\varepsilon_i\pm\varepsilon_j,\ i<j;\ \eps_i\}$
and $\ft\subset\fb\subset\fsp(V)$ corresponding to the choice of positive roots
$R^+=\{\delta_i\pm\delta_j,\ i<j;\ 2\delta_i\}$. We set $\rho_0=\frac12\sum_{\alpha\in R_0^+}\alpha$
and $\rho_1=\frac12\sum_{\alpha\in R^+}\alpha$.
We denote by $\Lambda_0$ (resp.\ $\Lambda_1$) the weight
lattice of $\SO(V_0)$ (resp.\ of $\Sp(V)$). We denote by $\Lambda_0^+\subset\Lambda_0$
(resp.\ $\Lambda_1^+\subset\Lambda_1$)
the monoids of dominant weights. For $\lambda\in\Lambda_0^+$ (resp.\ $\lambda\in\Lambda_1^+$)
we denote by $V_\lambda$ the irreducible representation of $\SO(V_0)$ (resp.\ of $\Sp(V)$)
with highest weight $\lambda$.

\subsection{The main theorem}
\label{main theorem}
We keep the notation of~\cite[\S2.2]{bft} with a single exception: we replace $V_1$ of {\em loc.cit.}\
with $V$. In particular, $\sG$ stands for the supergroup $\on{SOSp}(V_0|V)$ with the even part
$\sG_{\bar0}=\SO(V_0)\times\Sp(V)$, and $\bg$ stands for its Lie superalgebra $\osp(V_0|V)$.
Also, $\fG^\bullet$ stands for the dg-algebra\footnote{We view $\bg_{\bar1}$ as an {\em odd} vector space,
so that $\Sym^\bullet(\bg_{\bar1}[-1])$ (with grading disregarded) is really a symmetric (infinite-dimensional)
algebra, not an exterior algebra.} $\Sym^\bullet(\bg_{\bar1}[-1])$ with trivial differential,
and $\Lambda$ stands for the exterior algebra $\Lambda(V_0\otimes V)$.
Finally, $D^b\Rep(\ul\sG)$ stands for the dg-category of finite dimensional complexes of $\ul\sG$-modules.

Our goal is the following

\begin{thm}
\label{main}
  \textup{(a)} There exists an equivalence of triangulated categories
  $\Phi\colon D^{\sG_{\bar0}}_\perf(\fG^\bullet)\iso\CalD\CW\modu^{G_\bO,\lc}$
  commuting with the convolution action of the monoidal spherical Hecke category
  \[D^{G^\vee}_\perf\big(\Sym^\bullet(\gvee[-2])\big)\otimes D^G_\perf\big(\Sym^\bullet(\fg[-2])\big)\cong
  \mathrm{D}\modu^{G_\bO,\lc}(\Gr_G)\otimes\mathrm{D}\modu_{1/2}^{G_\bO,\lc}(\Gr_G).\]

\textup{(b)} The composed equivalence
  \[\Phi\circ\varkappa\colon
  D_\fid^{\sG_{\bar0}}(\Lambda)\iso\CalD\CW\modu^{G_\bO,\lc}\]
  is exact with respect to the tautological $t$-structure on
  $D_\fid^{\sG_{\bar0}}(\Lambda)$ and the tautological $t$-structure on $\CalD\CW\modu^{G_\bO,\lc}$.

  \textup{(c)} This equivalence is monoidal with respect to the tensor structure on
  $D_\fid^{\sG_{\bar0}}(\Lambda)$ and the fusion $\star$ on $\CalD\CW\modu^{G_\bO,\lc}$.

  \textup{(d)} The equivalence of \textup{(b)} extends to a monoidal equivalence from
  $SD_\fid^{G_{\bar0}}(\Lambda)=D^b\Rep(\ul\sG)$ to
  $S\CalD\CW\modu^{G_\bO,\lc}:=\CalD\CW\modu^{G_\bO,\lc}\otimes_{\on{Vect}}\on{SVect}$.

  \textup{(e)} The category $\CalD\CW\modu^{G_\bO,\lc}$ is equivalent to the
  dg-category of bounded complexes of the abelian category $\CalD\CW\modu^{G_\bO,\lc,\heartsuit}$
  (the heart of the tautological $t$-structure).

\end{thm}

The proof will be given in~\S\ref{proof}.\footnote{We take this opportunity to correct~\cite[Theorem 2.2.1.(f)]{bft}. Namely, the argument
in~\cite[\S2.9]{bft} only proves that for $\CF\in D_{\SO(N-1,\bO)}^{\on{comp}}({\mathbf{Gr}}_{\SO_N})$ we have
$\dim H^\bullet_{\SO(N-1,\bO)}({\mathbf{Gr}}_{\SO_N},\CF)<\infty$, but the converse implication is not true.}

\subsection{Irreducibles in $\CalD\CW\modu^{G_\bO,\lc,\heartsuit}$}
\label{2.6}
We identify $\Gr_G$ with the moduli space of parahoric subalgebras in $\fg_\bF$ conjugate to
$\fg_\bO$, and we identify the cotangent bundle $T^*\Gr_G$ with the space of pairs $(\fp,x)$ where
$\fp$ is a parahoric subalgebra, and $x$ lies in the nilpotent radical of $\fp$.
We identify $\fg_\bO^*$ with $\fg_\bF/t\fg_\bO$. The moment map
$\bmu\colon V_\bF\times T^*\Gr_G\to\fg_\bO^*$ of the $G_\bO$-action takes $(v,\fp,x)$ to
$tS^2v+x\pmod{t\fg_\bO}$. Here we identify $\fg_\bF$ with $\Sym^2V_\bF$. Thus the zero level
$\bmu^{-1}(0)$ of the moment map is the space of triples $(v,\fp,x)$ such that $tS^2v+x\in t\fg_\bO$.

\begin{lem}
  \label{lagrangian}
  The zero level $\bmu^{-1}(0)\subset V_\bF\times T^*\Gr_G$ of the moment map is Lagrangian.
\end{lem}

\begin{proof}
We consider the auxiliary mirabolic affine Grassmannian $V_\bF\times\Gr_{\GL(2n)}$ and its cotangent
bundle $T^*V_\bF\times T^*\Gr_{\GL(2n)}\cong V_\bF\times V_\bF\times T^*\Gr_{\GL(2n)}$. It is equipped
with the moment map of the $\GL(2n,\bO)$-action
$\bmu_{\GL(2n)}\colon T^*V_\bF\times T^*\Gr_{\GL(2n)}\to\fgl(2n,\bO)^*$. The zero level
$\bmu_{\GL(2n)}^{-1}(0)$ is Lagrangian since the set of $\GL(2n,\bO)$-orbits
in the mirabolic affine Grassmannian $V_\bF\times\Gr_{\GL(2n)}$ is discrete.

We have an involution $\sigma\circlearrowright\GL(2n)$ whose fixed point set is
$\Sp(2n)\subset\GL(2n)$. It induces an involution of $\Gr_{\GL(2n)}$ and of $T^*\Gr_{\GL(2n)}$,
and we extend it to the same named involution
$\sigma\circlearrowright(T^*V_\bF\times T^*\Gr_{\GL(2n)}\cong V_\bF\times V_\bF\times T^*\Gr_{\GL(2n)})$
permuting the two factors $V_\bF$. The moment map $\bmu_{\GL(2n)}$ is $\sigma$-equivariant, and
the $\sigma$-fixed point set of $\bmu_{\GL(2n)}\colon T^*V_\bF\times T^*\Gr_{\GL(2n)}\to\fgl(2n,\bO)^*$
is nothing but $\bmu\colon V_\bF\times T^*\Gr_G\to\fg_\bO^*$.

Now since the zero level $\bmu_{\GL(2n)}^{-1}(0)$ is Lagrangian (and hence isotropic), the zero
level $\bmu^{-1}(0)$ is isotropic as well. Finally, $\bmu^{-1}(0)$ is coisotropic since
$V_\bF\times T^*\Gr_G$ is isomorphic to the cotangent bundle $T^*(V_\bF/V_\bO\times\Gr_G)$.
This isomorphism is {\em not} $G_\bO$-equivariant, but $\bmu^{-1}(0)$ is the union of torsors
over conormal bundles to the {\em relevant}\footnote{For the notion of relevant orbits and torsors over their conormal bundles see
e.g.~\cite[\S4.3 and Remark 4.3.2]{fgt2}.} $G_\bO$-orbits in $V_\bF/V_\bO\times\Gr_G$, hence
coisotropic.
\end{proof}

Recall that the irreducible components of $\bmu_{\GL(2n)}^{-1}(0)$ are indexed by {\em bisignatures}
$(\blambda,\bnu)=(\lambda_1\geq\ldots\geq\lambda_{2n},\ \nu_1\geq\ldots\geq\nu_{2n})$: a component
is the closure of the conormal bundle to the corresponding $\GL(2n,\bO)$-orbit in
$V_\bF\times\Gr_{\GL(2n)}$~\cite[Proposition 8]{fgt}. The component indexed by $(\blambda,\bnu)$
will be denoted $\Lambda^{(\blambda,\bnu)}_{\GL(2n)}$. We identify $T^*V_\bF\times T^*\Gr_{\GL(2n)}$ (the Schr\"odinger model) with
the Heisenberg model $T^*(V_\bF/V_\bO\oplus V_\bF/V_\bO)\times T^*\Gr_{\GL(2n)}$, so that $\Lambda^{(\blambda,\bnu)}_{\GL(2n)}$
is the closure of a torsor over the conormal bundle to a relevant $\GL(2n,\bO)$-orbit in
$(V_\bF/V_\bO\oplus V_\bF/V_\bO)\times\Gr_{\GL(2n)}$, as in the end of proof of~Lemma~\ref{lagrangian}.
We denote this torsor over the conormal bundle by $\oLambda^{(\blambda,\bnu)}_{\GL(2n)}$ (a smooth dense
open subvariety in $\Lambda^{(\blambda,\bnu)}_{\GL(2n)}$). For a signature
$\blambda=(\lambda_1\geq\ldots\geq\lambda_{2n})$ we denote by $\blambda^*$ the signature
$(-\lambda_{2n}\geq\ldots\geq-\lambda_1)$. Recall the involution $\sigma$ introduced in the proof
of~Lemma~\ref{lagrangian}. The next lemma follows immediately from definitions:

\begin{lem}
  \label{sigma}
  The involution $\sigma$ takes $\Lambda^{(\blambda,\bnu)}_{\GL(2n)}$ to
  $\Lambda^{(\blambda^*,\bnu^*)}_{\GL(2n)}$, and $\oLambda^{(\blambda,\bnu)}_{\GL(2n)}$ to
  $\oLambda^{(\blambda^*,\bnu^*)}_{\GL(2n)}$. \hfill $\Box$
\end{lem}

The set of self-dual signatures $\blambda=\blambda^*$ is in bijection with the set of length $n$
partitions $\btheta=(\theta_1\geq\ldots\geq\theta_n\geq0)$; namely,
$\btheta\mapsto\blambda:=(\theta_1\geq\ldots\geq\theta_n\geq-\theta_n\geq\ldots\geq-\theta_1)$.
Given partitions $\bzeta=(\zeta_1\geq\ldots\geq\zeta_n\geq0)$ and $\btheta=(\theta_1\geq\ldots\geq\theta_n\geq0)$
we consider the corresponding self-dual bisignature
\[(\blambda,\bnu):=(\zeta_1\geq\ldots\geq\zeta_n\geq-\zeta_n\geq\ldots\geq-\zeta_1,\
\theta_1\geq\ldots\geq\theta_n\geq-\theta_n\geq\ldots\geq-\theta_1)\]
and the corresponding $\GL(2n,\bO)$-orbit $\BO_{(\blambda,\bnu)}$ in $V_\bF\times\Gr_{\GL(2n)}$.

\begin{lem}
  \label{irreducible}
  The fixed point set of $\sigma$ in 
  $\oLambda^{(\blambda,\bnu)}_{\GL(2n)}$ is irreducible.
\end{lem}

\begin{proof}
Recall that $\oLambda^{(\blambda,\bnu)}_{\GL(2n)}$ is a torsor over the conormal bundle to a certain relevant
orbit $\BO_{(\blambda,\bnu)}^{\on{Heis}}\subset (V_\bF/V_\bO\oplus V_\bF/V_\bO)\times\Gr_{\GL(2n)}$ in the Heisenberg model.
The orbit $\BO_{(\blambda,\bnu)}^{\on{Heis}}$ projects onto the $\GL(2n,\bO)$-orbit $\Gr^{\blambda+\bnu}_{\GL(2n)}$
in the affine Grassmannian $\Gr_{\GL(2n)}$. Any fiber of this projection is open in an appropriate
lattice in $V_\bF$. Hence any fiber of the projection $\varPi$ of
$\oLambda^{(\blambda,\bnu)}_{\GL(2n)}$ to $\Gr^{\blambda+\bnu}_{\GL(2n)}$ is open in
an appropriate affine space.
Now the fixed point set of $\sigma$ in $\Gr^{\blambda+\bnu}_{\GL(2n)}$ is the
corresponding $\Sp(2n,\bO)$-orbit $\Gr^{\bzeta+\btheta}_G$ in $\Gr_G$; in particular, it is irreducible.
The fixed point set of $\sigma$ in the fiber of $\varPi$ over a point of $\Gr^{\bzeta+\btheta}_G$
is irreducible as well. The irreducibility of the fixed point set
$\big(\oLambda^{(\blambda,\bnu)}_{\GL(2n)}\big)^\sigma$ follows.
\end{proof}

We define an irreducible component $\Lambda^{(\bzeta,\btheta)}\subset\bmu^{-1}(0)$ as the closure
of the fixed point set $\big(\oLambda^{(\blambda,\bnu)}_{\GL(2n)}\big)^\sigma$
of~Lemma~\ref{irreducible}.

\begin{prop}
  \label{classification}
  Any irreducible component of $\bmu^{-1}(0)$ is of the form $\Lambda^{(\bzeta,\btheta)}$ for a
  pair of length $n$ partitions $(\bzeta,\btheta)$.
\end{prop}

The proof follows from the part (c) of the next elementary linear algebraic lemma. \hfill $\Box$

\begin{lem}
  \label{elementary}
  \textup{(a)} Let $\Lambda\subset U$ be a Lagrangian subspace of a symplectic vector space $U$. Let
  $\sigma\circlearrowright U$ be an involution respecting the symplectic form, but not necessarily
  preserving $\Lambda$. Suppose the invariant subspace $\Lambda^\sigma$ is Lagrangian in the
  symplectic vector space $U^\sigma$. Then $\sigma(\Lambda)=\Lambda$.

  \textup{(b)} Conversely, if $\sigma(\Lambda)=\Lambda$, then $\Lambda^\sigma$ is Lagrangian in $U^\sigma$.

  \textup{(c)} Let $\Lambda\subset X$ be a smooth locally closed irreducible Lagrangian subvariety of a
  symplectic variety $X$. Let $\sigma\circlearrowright X$ be an involution respecting the symplectic form
  and preserving $\Lambda$. Then the fixed point subvariety $\Lambda^\sigma$ is Lagrangian in $X^\sigma$.
  \hfill $\Box$
\end{lem}

\bigskip

We can view irreducibles in $\CalD\CW\modu^{G_\bO,\lc,\heartsuit}$ as irreducible $G_\bO$-equivariant
$D$-modules on $(V_\bF/V_\bO)\times\Gr_G$. Proposition~\ref{classification} implies that among
the uncountably many $\Sp(2n,\bO)$-orbits in $(V_\bF/V_\bO)\times\Gr_G$,
only those orbits whose conormal bundle closures are of type $\Lambda^{(\bzeta,\btheta)}$
(so called {\em relevant} orbits) carry irreducibles in $\CalD\CW\modu^{G_\bO,\lc,\heartsuit}$.
For example, let $E_0\in\CalD\CW\modu^{G_\bO,\lc,\heartsuit}$ stand for the irreducible module
$\IC_0\otimes\BC[V_\bO]$. Then the conormal bundle to its support is the irreducible component
$\Lambda^{(0,0)}$. More generally, we view a length $n$ partition $\btheta$ (resp.\ $\bzeta$) as a
dominant weight of $\Sp(V)$ (resp.\ of $\SO(V_0)$), and consider the corresponding irreducible
$D$-modules $\IC_\btheta\in\mathrm{D}\modu_{1/2}^{G_\bO,\lc,\heartsuit}(\Gr_G)$
(resp.\ $\IC_\bzeta\in\mathrm{D}\modu^{G_\bO,\lc,\heartsuit}(\Gr_G)$)
corresponding to $V_\btheta\in\Rep(\Sp(V))$ (resp.\ to $V_\bzeta\in\Rep(\SO(V_0))$) under the
(twisted) Satake equivalence of~\S\ref{satake}. Then the support of
$\IC_\bzeta*E_0*\IC_\btheta$ is the closure of the relevant $\Sp(2n,\bO)$-orbit
$\BO_{(\bzeta,\btheta)}\subset(V_\bF/V_\bO)\times\Gr_G$
whose conormal bundle closure is the component $\Lambda^{(\bzeta,\btheta)}$.

\begin{prop}
  \label{irreducibles}
For any length $n$ partitions $\bzeta,\btheta$, there is a unique irreducible module $\IC_{(\bzeta,\btheta)}$
in $\CalD\CW\modu^{G_\bO,\lc,\heartsuit}$ whose support is the closure of $\BO_{(\bzeta,\btheta)}$.
\end{prop}

\begin{proof}
We have to check that the stabilizer in $\Sp(2n,\bO)$ of a point in $\BO_{(\bzeta,\btheta)}$ is connected.
The argument is similar to the proof of~\cite[Lemma 2.3.3]{bft}.
Namely, a representative of the orbit $\BO_{(\bzeta,\btheta)}$ is given by
\begin{multline*}
  v_{\bzeta,\btheta}:=t^{-\zeta_1}e_1+\ldots+t^{-\zeta_n}e_n\in V_\bF/V_\bO;\\
L_{\bzeta,\btheta}:=\bO t^{-\zeta_1-\theta_1}e_1\oplus\ldots\oplus\bO t^{-\zeta_n-\theta_n}e_n\oplus
\bO t^{\zeta_n+\theta_n}e_{n+1}\oplus\ldots\oplus\bO t^{\zeta_1+\theta_1}e_{2n}\in \Gr_G.
\end{multline*}
Here we chose a basis $e_1,\ldots,e_{2n}$ of $V$ such that $\langle e_i,e_{2n+1-j}\rangle=\delta_{ij}$
for $1\leq i,j\leq n$. 

First, we consider $\on{Stab}_{\Sp(2n,\bO)}(L_{\bzeta,\btheta})$.
Note that $(L_{\bzeta,\btheta},v_{\bzeta,\btheta})$ is fixed by the action of the one-parametric
subgroup $\BC^\times_{\bzeta,1}$ in the product of the Cartan torus of $G$ and the loop
rotation $\BC^\times$. Moreover, the conjugation action of $\BC^\times_{\bzeta,1}$ contracts
$\on{Stab}_{\Sp(2n,\bO)}(L_{\bzeta,\btheta})$ to $Q_{\bzeta,\btheta}$: a subgroup isomorphic to the
intersection in $G$ of the negative parabolic $P^-_\bzeta$ corresponding to the dominant weight $\bzeta$
and the positive parabolic $P^+_\btheta$ corresponding to the dominant weight $\btheta$.
(Note that $Q_{\bzeta,\btheta}$ needs {\em not} be reductive.) The attractor to the neutral element of $Q_{\bzeta,\btheta}$
is a unipotent subgroup $U\subset\on{Stab}_{\Sp(2n,\bO)}(L_{\bzeta,\btheta})$, and
$\on{Stab}_{\Sp(2n,\bO)}(L_{\bzeta,\btheta})=Q_{\bzeta,\btheta}\ltimes U$.
Now we have to find the stabilizer of $v_{\bzeta,\btheta}$ in $\on{Stab}_{\Sp(2n,\bO)}(L_{\bzeta,\btheta})$.
Due to the contraction property above, this stabilizer is the semidirect product of
$\on{Stab}_{Q_{\bzeta,\btheta}}(v_{\bzeta,\btheta})$ and $\on{Stab}_U(v_{\bzeta,\btheta})$.
The stabilizer $\on{Stab}_U(v_{\bzeta,\btheta})$ is connected, and it remains to check that the
stabilizer $H:=\on{Stab}_{Q_{\bzeta,\btheta}}(v_{\bzeta,\btheta})$ is connected.

Note that $H$ is connected if its intersection with any Levi subgroup (i.e.\ a maximal reductive subgroup)
$L\subset Q_{\bzeta,\btheta}$ is connected. Indeed, take $h\in H$.
Take the Jordan-Chevalley decomposition $h=h_sh_u$. Then $h_u\in H^\circ$, and $h_s$ lies in some Levi subgroup
$L$, but $L\cap H$ is connected by our assumption. Therefore, $h$ lies in the neutral component of $H$.

It remains to check the connectedness property of the intersection of $H$ with any Levi subgroup of
$Q_{\bzeta,\btheta}$. The reductive quotient $Q_{\bzeta,\btheta}^{\on{red}}$ of $Q_{\bzeta,\btheta}$ is isomorphic to the
following Levi subgroup $L$ of $G$. We write the partition $\bzeta+\btheta$ in the form $(0^{m_0}1^{m_1}\ldots)$
(where almost all multiplicities $m_i$ are zero). Then $L\simeq\Sp(2m_0)\times\prod_{i>0}\GL(m_i)$.
An arbitrary Levi subgroup $L'\subset Q_{\bzeta,\btheta}$ is a conjugate $q^{-1}Lq$ for some $q\in Q_{\bzeta,\btheta}$.
We have to check that $L\cap \on{Stab}_{Q_{\bzeta,\btheta}}(qv_{\bzeta,\btheta})=\on{Stab}_L(qv_{\bzeta,\btheta})$ is
connected for any $q\in Q_{\bzeta,\btheta}$. Now $qv_{\bzeta,\btheta}\in\BC t^{-\zeta_1}e_1\oplus\ldots+\BC t^{-\zeta_n}e_n$.

The stabilizer $\on{Stab}_L(qv_{\bzeta,\btheta})$ is the product of stabilizers of certain vectors (summands of
$qv_{\bzeta,\btheta}$, like e.g.\ $t^{-\zeta_1}e_1+t^{-\zeta_2}e_2+\ldots+t^{-\zeta_{m_{i_{\on{max}}}}}e_{m_{i_{\on{max}}}}$) in the GL-factors of
$L\simeq\Sp(2m_0)\times\prod_{i>0}\GL(m_i)$. Finally, the stabilizer of a nonzero vector in
$\GL(m_i)$ is a connected mirabolic subgroup (and the stabilizer of a zero vector is the connected group $\GL(m_i)$ itself).
\end{proof}

\begin{rem}
  In fact, we will see that
  $\IC_\bzeta*E_0*\IC_\btheta\simeq\IC_{(\bzeta,\btheta)}$ in~Corollary~\ref{irreducibility} below.
\end{rem}

\subsection{De-equivariantized Ext algebra}
\label{deeq}
Recall the notation of the end of~\S\ref{sosp}. For dominant weights $\lambda_0\in\Lambda_0^+$ and
$\lambda_1\in\Lambda_1^+$, we denote by $\IC_{\lambda_0}\in\mathrm{D}\modu^{G_\bO,\lc,\heartsuit}(\Gr_G)$
(resp.\ by $\IC_{\lambda_1}\in\mathrm{D}\modu_{1/2}^{G_\bO,\lc,\heartsuit}(\Gr_G)$) the irreducible $D$-modules
corresponding to $V_{\lambda_0}\in\Rep(\SO(V_0))$ (resp.\ to $V_{\lambda_1}\in\Rep(\Sp(V))$) under the
(twisted) Satake equivalence of~\S\ref{satake}.

Also, recall that $E_0\in\CalD\CW\modu^{G_\bO,\lc,\heartsuit}$ stands for the irreducible module
$\IC_0\otimes\BC[V_\bO]$.

We restrict the action
\[\mathrm{D}\modu^{G_\bO,\lc}(\Gr_G)\otimes
\mathrm{D}\modu_{1/2}^{G_\bO,\lc}(\Gr_G)\circlearrowright\CalD\CW\modu^{G_\bO,\lc}\]
of~\S\ref{mirab} to the action of \[\mathrm{D}\modu^{G_\bO,\lc,\heartsuit}(\Gr_G)\otimes
\mathrm{D}\modu_{1/2}^{G_\bO,\lc,\heartsuit}(\Gr_G)\cong\Rep(\SO(V_0))\otimes\Rep(\Sp(V)).\]
Let $\CalD\CW\modu^{G_\bO,\deeq}$ denote the corresponding de-equivariantized category, see~\cite{1-aff}.
We have
\begin{multline}
  \label{def deeq}
  \RHom_{\CalD\CW\modu^{G_\bO,\deeq}}(\CF,\CG)\\
  =\bigoplus_{\lambda_0\in\Lambda_0^+,\ \lambda_1\in\Lambda_1^+}\RHom_{\CalD\CW\modu^{G_\bO,\lc}}
  (\CF,\IC_{\lambda_0}*\CG*\IC_{\lambda_1})\otimes V_{\lambda_1}^*\otimes V_{\lambda_0}^*.
\end{multline}

\begin{lem}
  \label{formality}
  The dg-algebra $\RHom_{\CalD\CW\modu^{G_\bO,\deeq}}(E_0,E_0)$ is formal, i.e.\ it is
  quasiisomorphic to the graded algebra $\Ext_{\CalD\CW\modu^{G_\bO,\deeq}}^\bullet(E_0,E_0)$ with trivial
  differential.
\end{lem}

\begin{proof}
  The argument is the same as in the proof of~\cite[Lemma 3.9.1]{bfgt}. The only difference is in
  the application of yoga of weights. To this end we apply the equivalence
  $\mathrm{D}\modu_{1/2}^{G_\bO\ltimes V_\bO\times\BG_a,\chi,\lc}(\Gr_G\times\on{Heis})\cong\CalD\CW\modu^{G_\bO,\lc}$
  of~Lemma~\ref{heis}. Then we lift all encountered objects to the mixed setting of polarizable twistor
  $D$-modules of~\cite{mo1,mo2} (this can be done as by~\cite[Theorem 1.4.4]{mo1}, all semisimple holonomic
  $D$-modules can be lifted to pure weight~$0$ objects in this mixed setting). 
  Now the sheaves $\IC_{\lambda_0}*E_0*\IC_{\lambda_1}$ are pure weight~$0$ by~\cite[Theorem 18.1.1]{mo1}.
  Furthermore, $\RHom_{\CalD\CW\modu^{G_\bO}}(E_0,\IC_{\lambda_0}*E_0*\IC_{\lambda_1})$ is calculated as follows.
  First, we take the hyperbolic restriction $\Phi_{\lambda_0,\lambda_1}$ of
  $\IC_{\lambda_0}*E_0*\IC_{\lambda_1}\in \mathrm{D}\modu_{1/2}^{G_\bO\ltimes V_\bO\times\BG_a,\chi,\lc}(\Gr_G\times\on{Heis})$
  with respect to the $\BC^\times$-action $c\cdot(g(t),v(t),a)=(g(c^2t),cv(c^2t),a)$.
  It is supported on $\Gr_G^{\BC^\times}\times\BG_a\subset\Gr_G\times\on{Heis}$
  (where $\Gr_G^{\BC^\times}$ stands for the fixed points of the loop rotation on $\Gr_G$), and is $\exp$-equivariant along $\BG_a$.
  So $\Phi_{\lambda_0,\lambda_1}\otimes\exp^{-1}$ descends to $\ol{\Phi}_{\lambda_0,\lambda_1}\in\mathrm{D}\modu_{1/2}^{G_\bO}(\Gr_G^{\BC^\times})$.
  The latter sheaf is pure of weight~$0$ due to the weight-preserving property of the hyperbolic restriction.
  Finally, $\RHom_{\CalD\CW\modu^{G_\bO}}(E_0,\IC_{\lambda_0}*E_0*\IC_{\lambda_1})=\ol{\Phi}{}^0_{\lambda_0,\lambda_1}$, 
  where $\ol{\Phi}{}^0_{\lambda_0,\lambda_1}$ is the direct summand of $\ol{\Phi}_{\lambda_0,\lambda_1}$ supported at the point
  $\Gr^0_G\subset\Gr_G$.
  
  The rest of the argument is identical to the one in the proof of~\cite[Lemma 3.9.1]{bfgt}.
\end{proof}

We denote the dg-algebra
$\Ext^\bullet_{\CalD\CW\modu^{G_\bO,\deeq}}(E_0,E_0)$ (with trivial differential)
by $\fE^\bullet$. Since it is an Ext-algebra in the de-equivariantized category between objects
induced from the original category, it is automatically
equipped with an action of $\SO(V_0)\times\Sp(V)=\sG_{\bar0}$ (notation of~\S\ref{main theorem}),
and we can consider the corresponding triangulated category $D^{\sG_{\bar0}}_\perf(\fE^\bullet)$.

\begin{lem}
  \label{purity}
There is a canonical monoidal equivalence $D^{\sG_{\bar0}}_\perf(\fE^\bullet)\iso\CalD\CW\modu^{G_\bO,\lc}$.
\end{lem}

\begin{proof}
  Same as the one of~\cite[Lemma 3.9.2]{bfgt}. For the monoidal property, note that $E_0$ is the unit
  object with respect to fusion; in particular, $E_0\star E_0=E_0$. Also, for $\CE\in\CalD\CW\modu^{G_\bO,\lc}$,
  $\CF_0\in\mathrm{D}\modu^{G_\bO,\heartsuit}(\Gr_G)$,
  $\CF_1\in\mathrm{D}\modu^{G_\bO,\heartsuit}_{1/2}(\Gr_G)$, we have natural isomorphisms
  $\CF_0*\CE\cong(\CF_0*E_0)\star\CE$, $\CE*\CF_1\cong\CE\star(E_0*\CF_1).$
\end{proof}

\subsection{Equivariant De Rham cohomology and universal objects}
\label{de rham'}
We will work with the category $\CalD'\CW\modu^{G'_\bO,\lc}\cong\CalD\CW\modu^{G_\bO,\lc}$
in the realization of~Remark~\ref{prime},
i.e.\ in the equivalent incarnation
$\mathrm{D}\modu^{G_\bO\ltimes V_\bO\times\BG_a,\chi,\lc}(\Gr_G\times\on{Heis})$.
Let $\imath\colon\Gr_G\times V_\bO\times\BG_a\hookrightarrow\Gr_G\times\on{Heis}$ denote the
closed embedding.

Let $\CM_\chi$ stand for a unique
$(G_\bO\ltimes V_\bO\times\BG_a,\chi)$-equivariant $D$-module on $V_\bO\times\BG_a$ that is
a free rank one $\CO_{V_\bO\times\BG_a}$-module.
For a $D$-module $\CF\in\mathrm{D}\modu^{G_\bO\ltimes V_\bO\times\BG_a,\chi,\lc}(\Gr_G\times\on{Heis})$,
its restriction $\imath^!\CF$ decomposes as $\imath^!\CF=\CF_0\boxtimes\CM_\chi$ for some
$\CF_0\in\mathrm{D}\modu^{G_\bO,\lc}(\Gr_G)$. We will denote $\imath^0\CF:=\CF_0$.
We will keep the same notation $\imath^0$ for the functor
$\CalD\CW\modu^{G_\bO,\lc}\to\mathrm{D}\modu^{G_\bO,\lc}(\Gr_G)$ obtained by composing $\imath^0$
with the equivalence $\CalD\CW\modu^{G_\bO,\lc}\cong\CalD'\CW\modu^{G'_\bO,\lc}\cong
\mathrm{D}\modu^{G_\bO\ltimes V_\bO\times\BG_a,\chi,\lc}(\Gr_G\times\on{Heis})$ of~Remark~\ref{prime}.
We set
\begin{multline}
  \label{cohom'}
  H^\bullet_{G_\bO,\DR}(-):=H^\bullet_{G_\bO,\DR}(\Gr_G,\imath^0-)\colon\\
  \CalD\CW\modu^{G_\bO,\lc}\cong\CalD'\CW\modu^{G'_\bO,\lc}\cong
  \mathrm{D}\modu^{G_\bO\ltimes V_\bO\times\BG_a,\chi,\lc}(\Gr_G\times\on{Heis})\to\BC[\Xi_\fg]\modu.
\end{multline}

We introduce the following object
\begin{equation}
  \label{u'}
  \CU':=\bomega_{\Gr_G}\boxtimes\BC[V_\bO]\in\CalD'\CW\modu^{G'_\bO}
\end{equation}
of the `big' ({\em not} renormalized) category $\CalD'\CW\modu^{G'_\bO}$.
Under the equivalence $\CalD'\CW\modu^{G'_\bO}\cong\CalD\CW\modu^{G_\bO}$ of~\S\ref{mirab},
the object $\CU'$ goes to the `quantization of the universal lattice' (see e.g.~\cite[\S1.3]{laf})
\begin{equation}
  \label{u}
  \CU\in\CalD\CW\modu^{G_\bO}.
\end{equation}

(The existence of $\CU$ also follows from the twisted action $\mathrm{D}\modu_{-1/2}(G_\bF)$ on $\CW\modu$ and the fact that $\BC[V_\bO]\in\CW\modu$
is $G_\bO$-equivariant.) For later use we formulate a relation between $\CU,\CU'$ and the convolution action of $D$-modules on $\Gr_G$
on the unit objects of $\CalD\CW\modu^{G_\bO,\lc},\CalD'\CW\modu^{G'_\bO,\lc}$. 

\begin{lem}
  \label{left action U}
  \textup{(a)} For $\CF_0\in\mathrm{D}\modu^{G_\bO,\lc}(\Gr_G)$, $E_0\in\CalD\CW\modu^{G_\bO,\lc}$ we have
  $\CF_0*E_0=\CF_0\stackrel{!}{\otimes}\CU\in\CalD\CW\modu^{G_\bO,\lc}$.

  \textup{(b)} For $\CF_1\in\mathrm{D}\modu_{1/2}^{G_\bO,\lc}(\Gr_G)$, $E_0\in\CalD\CW\modu^{G_\bO,\lc}$ we have\\
  $E_0*\CF_1=\CF_1\boxtimes\BC[V_\bO]=\CF_1\stackrel{!}{\otimes}\CU'\in\CalD\CW\modu^{G_\bO,\lc}$.

  \textup{(c)} For $\CG_0\in\mathrm{D}\modu^{G_\bO,\lc}(\Gr_G)$, $E'_0\in\CalD'\CW\modu^{G'_\bO,\lc}$ we have\\
  $E'_0*\CG_0=\CG_0\boxtimes\BC[V_\bO]=\CG_0\stackrel{!}{\otimes}\CU'\in\CalD'\CW\modu^{G'_\bO,\lc}$.

  \textup{(d)} For $\CG_1\in\mathrm{D}\modu_{-1/2}^{G_\bO,\lc}(\Gr_G)$, $E'_0\in\CalD'\CW\modu^{G'_\bO,\lc}$ we have\\
  $\CG_1*E'_0=\CG_1\stackrel{!}{\otimes}\CU\in\CalD'\CW\modu^{G'_\bO,\lc}$.
\end{lem}

\begin{proof}
To prove (a), we invoke the `symmetric' realization
$\CalD\CW\modu^{G_\bO}\cong
(\mathrm{D}\modu(\Gr_G)\otimes\mathrm{D}\modu_{1/2}(\Gr_G)\otimes\CW\modu)^{G_\bF}$ of~\eqref{two real}.
Under this equivalence, $E_0\in\CalD\CW\modu^{G_\bO}$ goes to
$\Delta_*\CU\in(\mathrm{D}\modu(\Gr_G)\otimes\mathrm{D}\modu_{1/2}(\Gr_G)\otimes\CW\modu)^{G_\bF}$,
where $\Delta\colon\Gr_G\hookrightarrow\Gr_G\times\Gr_G$ stands for the diagonal embedding.
Also, $\CF_0*E_0$ goes to $\on{pr}_1^!\CF_0\stackrel{!}{\otimes}\Delta_*\CU$.

(b) is clear, and (c,d) follow from (a,b) by the explanation in the last paragraph of~\S\ref{mirab}.
\end{proof}

\begin{rem}
Let $\CU_{\on{Heis}}$ denote the image of $\CU$ under the equivalence 
$\CalD\CW\modu^{G_\bO}\cong\mathrm{D}\modu_{1/2}^{G_\bO\ltimes V_\bO\times\BG_a,\chi}(\Gr_G\times\on{Heis})$,
cf.~Lemma~\ref{heis}.
Let also $\bomega_{\on{Heis}}$ denote the dualizing sheaf on $\on{Heis}$.
Then for $\CF_0\in\mathrm{D}\modu^{G_\bO,\lc}(\Gr_G)$, the convolution
action $\CF_0*E_0\in\CalD\CW\modu^{G_\bO,\lc}$ goes under the equivalence
$\CalD\CW\modu^{G_\bO,\lc}\cong\mathrm{D}\modu_{1/2}^{G_\bO\ltimes V_\bO\times\BG_a,\chi,\lc}(\Gr_G\times\on{Heis})$
of~Lemma~\ref{heis} to $(\CF_0\boxtimes\bomega_{\on{Heis}})\stackrel{!}{\otimes}\CU_{\on{Heis}}
  \in\mathrm{D}\modu_{1/2}^{G_\bO\ltimes V_\bO\times\BG_a,\chi,\lc}(\Gr_G\times\on{Heis})$.
\end{rem}

We have a natural morphism of dg-algebras
$R\Gamma_{G_\bO,\DR}(\Gr_G,\CO_{\Gr_G})\to
\on{RHom}_{\mathrm{D}\modu^{G_\bO}(\Gr_G)}(\bomega_{\Gr_G},\bomega_{\Gr_G})$.
Since the dg-algebra $R\Gamma_{G_\bO,\DR}(\Gr_G,\CO_{\Gr_G})$ is formal and quasi-isomorphic to
$\BC[\Xi_\fg]$, we obtain the morphisms of dg-algebras
\begin{equation}
  \label{action}
  \BC[\Xi_\fg]\to\on{RHom}_{\mathrm{D}\modu^{G_\bO}(\Gr_G)}(\bomega_{\Gr_G},\bomega_{\Gr_G})\ \on{and}\
    \BC[\Xi_\fg]\to\on{RHom}_{\CalD'\CW\modu^{G'_\bO}}(\CU',\CU').
\end{equation}

Also, let $\nabla\in(\CW\hat\otimes\CW)\modu$ be the diagonal module
(the image of the regular bimodule $\CW\in(\CW\hat\otimes\CW^{\on{opp}})\modu$ under the isomorphism
$\CW^{\on{opp}}\iso\CW$ that multiplies all the generators in $V_\bF$ by $\sqrt{-1}$).
Given $\CF\in\CalD'\CW\modu^{G'_\bO,\lc},\
\CG\in\CalD'\CW\modu^{G'_\bO}$, we define a pairing with values in $\BC[\Xi_\fg]\modu$:
\[\langle\CF,\CG\rangle=\Hom_{\CW\mathbin{\hat\otimes}\CW}(\nabla,H^\bullet_{G_\bO,\DR}(\Gr_G,\Delta^!(\CF\mathbin{\hat\boxtimes}\CG))).\]
Here $\Delta\colon\Gr_G\hookrightarrow\Gr_G\times\Gr_G$ is the diagonal embedding, so that
$H^\bullet_{G_\bO,\DR}(\Gr_G,\Delta^!(\CF\mathbin{\hat\boxtimes}\CG))$ carries a structure of a module over
$H^\bullet_{G_\bO}(\Gr_G)\otimes\CW\mathbin{\hat\otimes}\CW$.

Then we have a natural isomorphism of $\BC[\Xi_\fg]$-modules
\begin{equation}
  \label{pairing}
 \forall \CE'\in\CalD'\CW\modu^{G'_\bO,\lc},\ H^\bullet_{G_\bO,\DR}(\CE')\cong\langle\CE',\CU'\rangle.
\end{equation}

\subsection{Hecke eigenproperty}
\label{hecke}
Recall the classical Satake monoidal equivalence~\cite{mv}
\[\sH^\vee:=H^\bullet_\DR(\Gr_G,-)\colon\mathrm{D}\modu^{G_\bO,\heartsuit}(\Gr_G)\to\Rep(G^\vee).\]
Also, recall the twisted Satake monoidal equivalence~\cite{lys}
\[\sH:=H^\bullet_\DR(\Gr_G,-\stackrel{!}{\otimes}\CR)\colon
\mathrm{D}\modu^{G_\bO,\heartsuit}_{1/2}(\Gr_G)\to\Rep(G).\]
Twisting by the determinant line bundle $\CalD$ defines a monoidal equivalence
$\mathrm{D}\modu^{G_\bO,\heartsuit}_{-1/2}(\Gr_G)\iso\mathrm{D}\modu^{G_\bO,\heartsuit}_{1/2}(\Gr_G)$.
We will keep the same notation for the composed twisted Satake monoidal equivalence
\[\sH:=H^\bullet_\DR(\Gr_G,-\otimes\CalD\stackrel{!}{\otimes}\CR)\colon
\mathrm{D}\modu^{G_\bO,\heartsuit}_{-1/2}(\Gr_G)\to\Rep(G).\]
For $\CF_0\in\mathrm{D}\modu^{G_\bO,\heartsuit}(\Gr_G)$ and
$\CF_1\in\mathrm{D}\modu^{G_\bO,\heartsuit}_{1/2}(\Gr_G)$, we have isomorphisms of $\BC[\Xi_\fg]$-modules
(notation of~\S\ref{theta})
\[H^\bullet_{G_\bO,\DR}(\Gr_G,\CF_0)=\bkappa^{\gvee}(\sH^\vee(\CF_0)\otimes\Sym^\bullet(\gvee[-2])),\ \on{and}\]
\[H^\bullet_{G_\bO,\DR}(\Gr_G,\CF_1\stackrel{!}{\otimes}\CR)=
\bkappa^\fg(\sH(\CF_1)\otimes\Sym^\bullet(\fg[-2])).\]

Recall the convolution action of~\S\ref{mirab}:
\[\mathrm{D}\modu_{-1/2}^{G_\bO,\lc}(\Gr_G)\otimes\mathrm{D}\modu^{G_\bO,\lc}(\Gr_G)\circlearrowright
\CalD'\CW\modu^{G'_\bO}.\]

We will also need the convolution of coherent sheaves on $\Xi_\fg$ with support finite over $\Sigma_\fg$.
Namely, recall that $\Xi_\fg$ is canonically identified with the tangent vector bundle of $\Sigma_\fg$.
Hence we have the addition map $\on{add}\colon\Xi_\fg\times_{\Sigma_\fg}\Xi_\fg\to\Xi_\fg$. Thus given coherent sheaves $\CA_1,\CA_2$
on $\Xi_\fg$ with support finite over $\Sigma_\fg$, we can consider $\on{add}_*(\CA_1\otimes_{\Sigma_\fg}\CA_2)$.
The corresponding operation on modules over $\BC[\Xi_\fg]$ will be denoted simply by $\CM_1\otimes_{\BC[\Sigma_\fg]}\CM_2$.

\begin{prop}
  \label{eigenval}
  \textup{(a)}
  For the unit object $E'_0\in\CalD'\CW\modu^{G'_\bO,\lc}$ and any $\CG_0\in\mathrm{D}\modu^{G_\bO,\heartsuit}(\Gr_G)$,
  $\CG_1\in\mathrm{D}\modu^{G_\bO,\heartsuit}_{-1/2}(\Gr_G)$, we have a natural isomorphism
  of $\BC[\Xi_\fg]$-modules
  \[H^\bullet_{G_\bO,\DR}(\CG_1*E'_0*\CG_0)
  \cong H^\bullet_{G_\bO,\DR}(\Gr_G,\CalD\otimes\CR\stackrel{!}{\otimes}\CG_1)
  \otimes_{\BC[\Sigma_\fg]}H^\bullet_{G_\bO,\DR}(\Gr_G,\CG_0).\]
  These isomorphisms are compatible with the monoidal structures of the (twisted) Satake equivalences.

  \textup{(b)} Recall an equivalent realization
  \[\CalD'\CW\modu^{G'_\bO,\lc}\cong\CalD\CW\modu^{G_\bO,\lc}\circlearrowleft
  \mathrm{D}\modu^{G_\bO,\heartsuit}(\Gr_G)\otimes\mathrm{D}\modu_{1/2}^{G_\bO,\heartsuit}(\Gr_G).\]
  Then equivalently, for the unit object
  $E_0\in\CalD\CW\modu^{G_\bO,\lc}$ and any $\CF_0\in\mathrm{D}\modu^{G_\bO,\heartsuit}(\Gr_G)$,
  $\CF_1\in\mathrm{D}\modu^{G_\bO,\heartsuit}_{1/2}(\Gr_G)$, we have a natural isomorphism
  of $\BC[\Xi_\fg]$-modules
  \[H^\bullet_{G_\bO,\DR}(\CF_0*E_0*\CF_1)
  \cong H^\bullet_{G_\bO,\DR}(\Gr_G,\CF_0)
  \otimes_{\BC[\Sigma_\fg]}H^\bullet_{G_\bO,\DR}(\Gr_G,\CR\stackrel{!}{\otimes}\CF_1).\]
\end{prop}

\begin{proof}
  We prove (a). We first consider the equivariant cohomology of $\CG_1*E'_0$. We denote by $\on{pr}$ the first projection
  $\Gr_G\times V_\bF/V_\bO\to\Gr_G$. Then by~Lemma~\ref{left action U}(d), $\CG_1*E'_0=\CU\stackrel{!}{\otimes}\on{pr}^!\CG_1$, where $\CU$ is the quantization of the universal
  lattice, see~\eqref{u}. Furthermore, according to~\eqref{pairing}, we have $H^\bullet_{G_\bO,\DR}(\CG_1*E'_0)=\langle\CG_1*E'_0,\CU'\rangle=
  \Hom_{\CW\hat\otimes\CW}(\nabla,H^\bullet_{G_\bO,\DR}(\Gr_G,\Delta^!(\CU\stackrel{!}{\otimes}\on{pr}^!\CG_1\mathbin{\hat\boxtimes}\CU')))$.
  But $\Hom_{\CW\hat\otimes\CW}(\nabla,\Delta^!(\CU\mathbin{\hat\boxtimes}\CU'))=\CalD\otimes\CR$ by the internal Hom construction of $\CR$~\cite[\S3.1]{bdfrt}, and hence
  $\Hom_{\CW\hat\otimes\CW}(\nabla,\Delta^!(\CU\stackrel{!}{\otimes}\on{pr}^!\CG_1\mathbin{\hat\boxtimes}\CU'))
  =\CalD\otimes\CR\stackrel{!}{\otimes}\CG_1$.
  We conclude that $H^\bullet_{G_\bO,\DR}(\CG_1*E'_0) \cong H^\bullet_{G_\bO,\DR}(\Gr_G,\CalD\otimes\CR\stackrel{!}{\otimes}\CG_1)$.
  Finally, the fact that the equivariant cohomology of convolution with $\CG_0\in\mathrm{D}\modu^{G_\bO,\heartsuit}(\Gr_G)$ is the convolution
  of the corresponding $\BC[\Xi_\fg]$-modules is a consequence of the monoidal property of the derived Satake equivalence~\cite{bf}.  
\end{proof}

In particular, taking \[\CF_1=\IC_{\lambda_1}\in\mathrm{D}\modu^{G_\bO,\heartsuit}_{1/2}(\Gr_G),\
\CF_0=\IC_{\lambda_0}\in\mathrm{D}\modu^{G_\bO,\heartsuit}(\Gr_G),\] we obtain

\begin{cor}
  \label{cohomol}
  For $\lambda_0\in\Lambda_0^+,\ \lambda_1\in\Lambda_1^+$, we have a canonical isomorphism
  of $\BC[\Xi_\fg]$-modules
  \[\bkappa^{\gvee}(V_{\lambda_0}\otimes\Sym^\bullet(\gvee[-2]))\otimes_{\BC[\Sigma_\fg]}\bkappa^\fg(V_{\lambda_1}\otimes\Sym^\bullet(\fg[-2]))\iso
  H^\bullet_{G_\bO,\DR}(\IC_{\lambda_0}*E_0*\IC_{\lambda_1}).\]
\end{cor}

We will also need the following generalization of~Corollary~\ref{cohomol}:

\begin{cor}
  \label{stronger}
  Consider an object $\CE\in\CalD\CW\modu^{G_\bO,\lc}$. 
  For $U_0\in\Rep(G^\vee)$, $U_1\in\Rep(G)$, there is a canonical isomorphism of $\BC[\Xi_\fg]$-modules
  \begin{multline*}
    H^\bullet_{G_\bO,\DR}(\Gr_G,{\mathbf S}^{\gvee}(U_0))
    \otimes_{\BC[\Sigma_\fg]}H^\bullet_{G_\bO,\DR}(\CE)\otimes_{\BC[\Sigma_\fg]}
    H^\bullet_{G_\bO,\DR}(\Gr_G,\CR\stackrel{!}{\otimes}{\mathbf S}^\fg(U_1))\\
    \cong \bkappa^{\gvee}(U_0\otimes\Sym^\bullet(\gvee[-2]))
    \otimes_{\BC[\Sigma_\fg]}H^\bullet_{G_\bO,\DR}(\CE)\otimes_{\BC[\Sigma_\fg]}
    \bkappa^\fg(U_1\otimes\Sym^\bullet(\fg[-2]))\\
    \iso H^\bullet_{G_\bO,\DR}\left({\mathbf S}^{\gvee}(U_0)*\CE*{\mathbf S}^\fg(U_1)\right)
  \end{multline*}
  (notation of~\S\ref{satake}). These isomorphisms are compatible with the monoidal structures of the (twisted) Satake equivalences.
  In other words, the functor $H^\bullet_{G_\bO,\DR}\colon \CalD\CW\modu^{G_\bO,\lc}\to \BC[\Xi_\fg]\modu$ respects the structure of
  module categories over $\Rep(G^\vee)\otimes\Rep(G)$.
\end{cor}

Due to~\eqref{pairing}, the corollary follows from the next

\begin{prop}
  For $\CG_0\in\mathrm{D}\modu^{G_\bO,\heartsuit}(\Gr_G)$,
  $\CG_1\in\mathrm{D}\modu^{G_\bO,\heartsuit}_{-1/2}(\Gr_G)$, we have a natural isomorphism
  of $\BC[\Xi_\fg]$-modules in $\CalD'\CW\modu^{G'_\bO}$
  \[\CG_1*\CU'*\CG_0\cong H^\bullet_{G_\bO,\DR}(\Gr_G,\CalD\otimes\CR\stackrel{!}{\otimes}\CG_1)
  \otimes_{\BC[\Sigma_\fg]}H^\bullet_{G_\bO,\DR}(\Gr_G,\CG_0)\otimes_{\BC[\Sigma_\fg]}\CU'.\]
These isomorphisms are compatible with the monoidal structures of the (twisted) Satake equivalences.
\end{prop}

\begin{proof}
Recall~\eqref{u'} that $\CU'=\bomega_{\Gr_G}\boxtimes\BC[V_\bO]$. 
  
  The category $\mathrm{D}\modu(\Sp(2n,\bF))\otimes\CW\modu$ is equipped with the left (diagonal) action of
  $\mathrm{D}\modu_{-1/2}(\Sp(2n,\bF))$ and the right action of $\mathrm{D}\modu(\Sp(2n,\bF))$ (this action is only on the left factor).
  We have the functor $H^\bullet_{\DR}(\Gr_G,-)=p_*\colon \mathrm{D}\modu(\Sp(2n,\bF))\otimes\CW\modu\to\CW\modu$ that
  commutes with the two actions (the right action on $\CW\modu$ is trivial). We also have the functor
  $p^!\colon \CW\modu\to\mathrm{D}\modu(\Sp(2n,\bF))\otimes\CW\modu$. We take the left and right invariants of $G_\bO$, and keep
  the names $p_*,p^!$ for the induced functors on the invariants. Note that the left and right invariants of $G_\bO$ in $\CW\modu$
  is $\CW\modu^{G_\bO}\otimes\mathrm{D}\modu^{G_\bO}(\pt)$ (since the right action is trivial.) Also note that $p^!\BC[V_\bO]=\CU'$.
  Hence the Hecke eigenproperty of $\CU'$ follows from the Hecke eigenproperty of
  $\BC[V_\bO]\otimes\BC\in\CW\modu^{G_\bO}\otimes\mathrm{D}\modu^{G_\bO}(\pt)$.

  The trivial $D$-module $\BC\in\mathrm{D}\modu^{G_\bO}(\pt)$ clearly has the Hecke eigenproperty with respect to the action of
  $\mathrm{D}\modu^{G_\bO,\heartsuit}(\Gr_G)$ :
  $\BC*\CG_0\cong H^\bullet_{G_\bO,\DR}(\Gr_G,\CG_0)\otimes_{\BC[\Sigma_\fg]}\BC$, and these
  isomorphisms are compatible with the monoidal structure of the Satake equivalence. Let us reformulate this
  property in the following fancy way. We consider the functor
  $L_{\gvee}\colon \mathrm{D}\modu^{G_\bO}(\Gr_G)\to\mathrm{D}\modu^{G_\bO}(\pt)$, $\CG_0\mapsto\CG_0*\BC=H^\bullet_{G_\bO}(\Gr_G,\CG_0)$,
  and its right adjoint functor $R_{\gvee}\colon \mathrm{D}\modu^{G_\bO}(\pt)\to\mathrm{D}\modu^{G_\bO}(\Gr_G)$.
  Then $R_{\gvee}(\BC)=\bomega_{\Gr_G}$, and $R_{\gvee}\circ L_{\gvee}(\CG_0)=\bomega_{\Gr_G}*\CG_0$.
  The monadic structure on $R_{\gvee}\circ L_{\gvee}$ arises from the ring structure of $\bomega_{\Gr_G}$.

It remains to check the Hecke eigenproperty of $\BC[V_\bO]\in\CW\modu^{G_\bO}$ with respect to the action of
  $\mathrm{D}\modu^{G_\bO,\heartsuit}_{-1/2}(\Gr_G)$.
  We consider the functor $L\colon \mathrm{D}\modu^{G_\bO}_{-1/2}(\Gr_G)\to\CW\modu^{G_\bO}$, $\CG_1\mapsto\CG_1*\BC[V_\bO]$,
  and its right adjoint functor $R\colon \CW\modu^{G_\bO}\to\mathrm{D}\modu^{G_\bO}_{-1/2}(\Gr_G)$. By the internal Hom
  construction of $\CR$~\cite[\S3.1]{bdfrt}, we have $R(\BC[V_\bO])=\CR$, and $R\circ L(\CG_1)=\CG_1*\CR$. The monadic structure
  on $R\circ L$ arises from the ring structure of $\CR$. According to~\cite[Theorem 3.5.1]{bdfrt}, under the equivalence
  $\mathrm{D}\modu^{G_\bO}_{-1/2}(\Gr_G)\cong\mathrm{D}\modu^{G^\vee_\bO}(\Gr_{G^\vee})$ (the composition of $\bbeta^\fg$ of~\S\ref{satake}
  and the untwisted derived Satake equivalence for $\mathrm{D}\modu^{G^\vee_\bO}(\Gr_{G^\vee})$), the ring object $\CR$ goes to the
  ring object $\bomega_{\Gr_{G^\vee}}$. So the desired Hecke eigenproperty of $\BC[V_\bO]\in\CW\modu^{G_\bO}$
  follows from the
  Hecke eigenproperty of $\BC\in\mathrm{D}\modu^{G_\bO}(\pt)$ discussed in the previous paragraph (up to swapping $G$ and $G^\vee$). 
\end{proof}

\subsection{Injectivity}

\begin{lem}
  \label{injective}
  For $\lambda_0\in\Lambda_0^+,\ \lambda_1\in\Lambda_1^+$, the natural morphism
  \begin{multline*}\Ext_{\CalD\CW\modu^{G_\bO,\lc}}(E_0,\IC_{\lambda_0}*E_0*\IC_{\lambda_1})\\
    \to\Hom_{\BC[\Xi_\fg]}(H^\bullet_{G_\bO,\DR}(E_0),
    H^\bullet_{G_\bO,\DR}(\IC_{\lambda_0}*E_0*\IC_{\lambda_1}))
  \end{multline*}
  is injective.
\end{lem}

\begin{proof}

It suffices to check that the natural morphism
\begin{multline*}\Ext_{\CalD\CW\modu^{G_\bO,\lc}}(E_0,\IC_{\lambda_0}*E_0*\IC_{\lambda_1})\\
    \to\Hom_{\BC[\Sigma_\fg]}(H^\bullet_{G_\bO,\DR}(E_0),
    H^\bullet_{G_\bO,\DR}(\IC_{\lambda_0}*E_0*\IC_{\lambda_1}))
\end{multline*}
(in the RHS we take Hom over the $G_\bO$-equivariant cohomology of the point) is injective.
As in the proof of~Proposition~\ref{eigenval}, we rewrite the LHS as
\begin{multline*}
  \Ext_{\CalD\CW\modu^{G_\bO,\lc}}(E_0,\IC_{\lambda_0}*E_0*\IC_{\lambda_1})\simeq H^\bullet_{G_\bO,\DR}
  (\Gr_G,\IC_{\lambda_0}\stackrel{!}{\otimes}\IC_{\lambda_1}\stackrel{!}{\otimes}\CR)\\
  \simeq\Ext_{\mathrm{D}\modu^{G_\bO}(\Gr_G)}(\IC_{\lambda_0},\IC_{\lambda_1}\stackrel{!}{\otimes}\CR).
\end{multline*}
So we have to check the injectivity of
\begin{multline*}
  \Ext_{\mathrm{D}\modu^{G_\bO}(\Gr_G)}(\IC_{\lambda_0},\IC_{\lambda_1}\stackrel{!}{\otimes}\CR)\\
\to\Hom_{\BC[\Sigma_\fg]}(H^\bullet_{G_\bO,\DR}(\Gr_G,\IC_{\lambda_0}),
H^\bullet_{G_\bO,\DR}(\Gr_G,\IC_{\lambda_1}\stackrel{!}{\otimes}\CR)).
\end{multline*}
Since $\IC_{\lambda_0}$ is pure, and $\IC_{\lambda_1}\stackrel{!}{\otimes}\CR$ is $!$-pure,
the desired injectivity follows from~\cite[Theorem 1.1.1]{g2}.
\end{proof}

\subsection{Some Invariant Theory}
The following proposition is proved similarly to~\cite[Proposition 2.8.3]{bft}.

\begin{prop}
  \label{inv the}
Given an $\SO(V_0)$-module $M$ and an $\Sp(V)$-module $M'$, the Kostant functors of~\S\ref{sosp}
induce an isomorphism
\begin{multline*}
\Hom_{\SO(V_0)\times\Sp(V)\ltimes\BC[V_0\otimes V]}\left(\BC[V_0\otimes V],
M\otimes\BC[V_0\otimes V]\otimes M'\right)\\
\iso\Hom_{\BC[\Xi_\fg]}\left(\BC[\Sigma_\fg],\bkappa^{\gvee}(M\otimes\Sym^\bullet(\gvee[-2]))
\otimes_{\BC[\Sigma_\fg]}\bkappa^\fg(M'\otimes\Sym^\bullet(\fg[-2]))\right).
\end{multline*}
Here we identify $\Xi_\fg$ and $\Xi_{\gvee}$ as in~\S\ref{sosp}, and we view $\Sigma_\fg$ as the
zero section of the tangent bundle $\Xi_\fg=T\Sigma_\fg$. \hfill $\Box$
\end{prop}

\subsection{The generators of $\fE^\bullet$}
\label{generators}
Let $\omega\in\Lambda_0^+$ (resp.\ $\varpi\in\Lambda_1^+$) be the first fundamental weight such
that $V_\omega$ (resp.\ $V_\varpi$) is the tautological representation $V_0$ of $G^\vee=\SO(V_0)$
(resp. $V$ of $G=\Sp(V)$). We will describe the objects $\IC_\omega*E_0$ and $E_0*\IC_\varpi$
of $\CalD\CW\modu^{G_\bO}$.

Along the Grassmannian $\Gr_G$ they are both supported on the minimal Schubert variety
$\ol\Gr{}_G^\omega$. If we identify the Weyl algebra $\CW$ with differential operators on
$V_\bF/V_\bO$, then along this factor, both $\IC_\omega*E_0$ and $E_0*\IC_\varpi$ are supported on
$(t^{-1}V_\bO)/V_\bO$.

More precisely, we consider a symplectic vector space $(t^{-1}V_\bO)/tV_\bO$
along with its projection $\on{pr}$ onto $(t^{-1}V_\bO)/V_\bO$. Then the minimal Schubert variety
$\ol\Gr{}_G^\omega$ possesses a resolution of singularities
$\widetilde\Gr{}_G^\omega\stackrel{\pi}{\longrightarrow}\ol\Gr{}_G^\omega$ formed by all
the collections of lines $\ell\subset(t^{-1}V_\bO)/V_\bO$ and Lagrangian subspaces
$L\subset(t^{-1}V_\bO)/tV_\bO$ such that $L\subset\on{pr}^{-1}(\ell)$. The projection
$p\colon \widetilde\Gr{}_G^\omega\to\BP((t^{-1}V_\bO)/V_\bO),\ (\ell,L)\mapsto\ell$, identifies
$\widetilde\Gr{}_G^\omega$ with a $\BP^1$-bundle over $\BP((t^{-1}V_\bO)/V_\bO)$, namely the
projectivization of a rank~2 vector bundle $\CO\oplus\CO(2)$. The exceptional divisor
${\mathsf E}\subset\widetilde\Gr{}_G^\omega$ (the preimage of the singular base point
$0\in\ol\Gr{}_G^\omega$) is the infinite section of the above $\BP^1$-bundle.
The pullback of the determinant line bundle $\pi^*\CalD\simeq(p^*\CO(2))({\mathsf E})$.

We denote by $X\subset\widetilde\Gr{}_G^\omega\times((t^{-1}V_\bO)/V_\bO)$ the `universal line' whose
fiber over $(\ell,L)\in\widetilde\Gr{}_G^\omega$ is $\ell\subset(t^{-1}V_\bO)/V_\bO$. The open subset
$X^\circ\subset X$ is the preimage of the $G_\bO$-orbit $\Gr{}_G^\omega\subset\widetilde\Gr{}_G^\omega$.
The orbit $\Gr{}_G^\omega$ is identified with the total space of the line bundle $\CO(2)$ over
$\BP((t^{-1}V_\bO)/V_\bO)$, and $X^\circ$ is identified with the pullback of $\CO(-1)$ to this total
space. Hence we have a regular function $f\colon X^\circ\to\BG_a$ whose restriction to
any fiber of the former (resp.\ latter) line bundle is linear (resp.\ quadratic).
This function $f$ has a simple pole along the preimage in $X$ of the exceptional divisor
${\mathsf E}\subset\widetilde\Gr{}_G^\omega$.
We consider the $D$-module $\exp(f)$ on $X^\circ$ with irregular singularity at the preimage
of ${\mathsf E}$.
Also, we have the smooth irreducible $\sqrt{\CalD}$-twisted $D$-module on $\Gr{}_G^\omega$ with
1-dimensional fibers to be denoted
$\sqrt{\CalD}$. We will keep the same notation for its pullback to $X^\circ$. We define an
irreducible $\sqrt{\CalD}$-twisted $D$-module $\CF^\circ_0$ on $X^\circ$ as
$\sqrt{\CalD}\otimes\exp(f)$. Finally, we view $\CF^\circ_0$ as a $\sqrt{\CalD}$-twisted $D$-module
on $\Gr{}_G^\omega\times((t^{-1}V_\bO)/V_\bO)$ and define $\CF_0$ as the minimal extension
of $\CF^\circ_0$ to $\ol\Gr{}_G^\omega\times((t^{-1}V_\bO)/V_\bO)$. Then $\IC_\omega*E_0$ is nothing but
$\CF_0$.

Furthermore, we have the smooth irreducible $\sqrt{\CalD}$-twisted $D$-module $\CF^\circ_1$ with
1-dimensional
fibers on $\Gr{}_G^\omega\times\{0\}\subset\Gr{}_G^\omega\times((t^{-1}V_\bO)/V_\bO)$ (it was denoted
$\sqrt{\CalD}$ in the previous paragraph). It extends cleanly
to a $\sqrt{\CalD}$-twisted $D$-module on $\widetilde\Gr{}_G^\omega\times\{0\}\subset X$, and
we define $\CF_1$ as the direct image of this $D$-module under the projection
$X\to\ol\Gr{}_G^\omega\times((t^{-1}V_\bO)/V_\bO)$. Then $E_0*\IC_\varpi$ is nothing but $\CF_1$.
In particular, the minimal extension $\CF_1$ of $\CF_1^\circ$ is clean.

Since $\Gr{}_G^\omega\times\{0\}$ is a smooth divisor in $X^\circ$, we have the canonical elements
$h\in\Ext^1(\CF^\circ_1,\CF^\circ_0),\ h^*\in\Ext^1(\CF^\circ_0,\CF^\circ_1)$. Since
$E_0*\IC_\varpi$ is the clean extension of $\CF^\circ_1$, the elements $h,h^*$ give rise to
the same named elements
\[h\in\Ext^1_{\CalD\CW\modu^{G_\bO}}(E_0*\IC_\varpi,\IC_\omega*E_0),\
h^*\in\Ext^1_{\CalD\CW\modu^{G_\bO}}(\IC_\omega*E_0,E_0*\IC_\varpi).\]
The composition $h^*\circ h\in\Ext^2_{\CalD\CW\modu^{G_\bO}}(E_0*\IC_\varpi,E_0*\IC_\varpi)$ is the
multiplication by the first Chern class of the normal bundle $\CN_{\Gr{}_G^\omega\times\{0\}/X^\circ}$
equal to $-\frac12 c_1(\CalD)$.

Since $\Ext^1_{\CalD\CW\modu^{G_\bO}}(E_0*\IC_\varpi,\IC_\omega*E_0)=
\Ext^1_{\CalD\CW\modu^{G_\bO}}(E_0,\IC_\omega*E_0*\IC_\varpi)$,
we obtain the subspace $h\otimes V_0^*\otimes V^*\cong h\otimes V_0\otimes V\subset\fE^1:=
\Ext_{\CalD\CW\modu^{G_\bO,\deeq}}^1(E_0,E_0)$ (notation of~\S\ref{deeq}).
We will denote this subspace simply by $V_0\otimes V$.

\subsection{Proof of Theorem~\ref{main}}
\label{proof}
Due to the results of~\S\ref{generators}, we obtain a homomorphism from the free tensor algebra
\[\phi^\bullet\colon T(\Pi(V_0\otimes V)[-1])\to\fE^\bullet:=\Ext_{\CalD\CW\modu^{G_\bO,\deeq}}^\bullet(E_0,E_0).\]

\begin{lem}
  The homomorphism $\phi^\bullet$ factors through the projection
  \[T(\Pi(V_0\otimes V)[-1])\twoheadrightarrow\Sym(\Pi(V_0\otimes V)[-1])=\fG^\bullet\]
  (notation of~\S\ref{main theorem}) and induces an isomorphism $\fG^\bullet\iso\fE^\bullet$.
\end{lem}
\begin{proof}
  The same as the proof of~\cite[Lemma 2.6.2]{bft}, granted~Proposition~\ref{inv the} along
  with~Lemma~\ref{injective} and~Corollary~\ref{cohomol}.
  More precisely, the identification of the multiplication in the deequivariantized Ext-algebras of the unit objects
  in the geometric and algebraic categories follows from the compatibility of isomorphisms of~Corollary~\ref{stronger}
  with the monoidal structures. 
\end{proof}

Now Theorem~\ref{main} is proved the same way as~\cite[Theorem 2.2.1]{bft} is proved
in~\cite[\S\S2.7,2.10]{bft}. Moreover, the following corollary is proved the same way
as~\cite[Corollary 2.6.3]{bft}.

\begin{cor}
  \label{irreducibility}
  For any $\lambda_0\in\Lambda_0^+,\ \lambda_1\in\Lambda_1^+$, the convolution
  $\IC_{\lambda_0}*E_0*\IC_{\lambda_1}$ lies in the heart $\CalD\CW\modu^{G_\bO,\lc,\heartsuit}$
  and coincides with the irreducible object $\IC_{(\lambda_0,\lambda_1)}$ (notation of~\S\ref{2.6}).
\end{cor}

\section{Classical simple Lie superalgebras with invariant even pairing}
\label{complements}

\subsection{Orthosymplectic Gaiotto conjectures}
\label{osp2n}
Given a positive integer $m\leq n$ we set $k=n-m$ and consider
a nilpotent element $e\in\fsp(2n)$ of Jordan type $(2m,1^{2k})$. We fix a maximal connected reductive
subgroup $\Sp(2k)$ in the centralizer of $e$ (i.e.\ an orthogonal decomposition
$V=\BC^{2n}=\BC^{2m}\oplus\BC^{2k}$ compatible with $e$). We choose an $\fsl_2$-triple $(e,h,f)$
in $\fsp(2n)$. In the theory of finite $W$-algebras,
the Slodowy slice $\CS_e$ is obtained as the Hamiltonian reduction of $\fsp(2n)$ with respect
to a certain unipotent subgroup $U(e)\subset\Sp(2n)$ with a character. The Lie algebra
$\on{Lie}U(e)=\fsp(2n)_{\leq-3/2}$ depends on a choice of Lagrangian subspace in the
$2k$-dimensional $-1$-eigenspace $\fsp(2n)_{-1}$ of $h$. This choice can not be made
$\Sp(2k)$-invariant, so we will use a modification of this construction.

Let $\fK(e)$ denote the kernel of the Killing pairing $(e,-)\colon\fsp(2n)_{\leq-2}\to\BC$.
We have an extension of nilpotent Lie algebras
\[0\to\fK(e)\to\fsp(2n)_{\leq-1}\to\mathfrak{Heis}(\fsp(2n)_{-1})\to0,\]
where $\mathfrak{Heis}(\fsp(2n)_{-1})$ stands for the Heisenberg Lie algebra of the symplectic
vector space $\fsp(2n)_{-1}\simeq\BC^{2k}$. In case $k=n-1$, we have
$\fsp(2n)_{\leq-1}\simeq\mathfrak{Heis}(\fsp(2n)_{-1})$.
We denote by $\on{U}_k$ the unipotent subgroup of $\Sp(2n)$ with the Lie algebra $\fsp(2n)_{\leq-1}$.


As in~\S\ref{weyl}, we consider the completed Weyl algebra $\CW(\fsp(2n)_{-1}\otimes\bF)=\CW(\bF^{2k})$.
There is a twisted action $\mathrm{D}\modu_{-1/2}(\Sp(2k,\bF))\circlearrowright\CW(\bF^{2k})\modu$.
Let $\kappa_b$ stand for the basic level, i.e.\ the bilinear form $\on{Tr}(X\cdot Y)$ on $\fsp(2n)$.
It corresponds to the determinant line bundle on $\Gr_G$ (the ample generator of the Picard group).
Given $c\in\BC^\times$ we consider the categories
$\CalD\CW(\bF^{2k})\modu_{1/2+c^{-1}}:=\mathrm{D}\modu_{1/2+c^{-1}}(\Gr_G)\otimes\CW(\bF^{2k})\modu$
(resp.\ $\CalD\CW(\bF^{2k})\modu_{c^{-1}}:=\mathrm{D}\modu_{c^{-1}}(\Gr_G)\otimes\CW(\bF^{2k})\modu$),
and the categories $\CalD\CW(\bF^{2k})\modu_{1/2+c^{-1}}^{\Sp(2k,\bO)\ltimes\on{U}_k(\bF),\lc}$
(resp.\ $\CalD\CW(\bF^{2k})\modu_{c^{-1}}^{\Sp(2k,\bO)\ltimes\on{U}_k(\bF),\lc}$)
of locally compact $\Sp(2k,\bO)\ltimes\on{U}_k(\bF)$-equivariant objects.

\bigskip

On the dual side, we consider the quantum groups
\[U_q(\osp(2k+1|2n))\ \on{and}\ U_q(\osp(2n+1|2k)),\ q=\exp(\pi\sqrt{-1}/c).\]
More precisely,
when $k=0$, we consider the modified covering quantum supergroup $U_q(\osp(1|2n))$ of~\cite{cflw}.
We denote by $\Rep_q(\on{SOSp}(2k+1|2n))$
(resp.\ $\Rep_q(\on{SOSp}(2n+1|2k))$) the dg-category of finite dimensional complexes of
$U_q(\osp(2k+1|2n))$-modules (resp.\ $U_q(\osp(2n+1|2k))$-modules).

\begin{conj}
  \label{gaiotto}
  \textup{(a)} For $c\not\in\BQ^\times$, the categories
  \[\CalD\CW(\bF^{2k})\modu_{1/2+c^{-1}}^{\Sp(2k,\bO)\ltimes\on{U}_k(\bF),\lc} \on{and}\
  \Rep_q(\on{SOSp}(2k+1|2n))\]
  are equivalent as braided tensor categories, and this equivalence is compatible with the
  tautological $t$-structures.

  \textup{(b)} For $c\not\in\BQ^\times$, the categories
  \[\CalD\CW(\bF^{2k})\modu_{c^{-1}}^{\Sp(2k,\bO)\ltimes\on{U}_k(\bF),\lc} \on{and}\
  \Rep_q(\on{SOSp}(2n+1|2k))\]
  are equivalent as braided tensor categories, and this equivalence is compatible with the
  tautological $t$-structures.
\end{conj}

\begin{rem}
Consider the extremal case $m=n$ (that is, $k=0$). We set
$q'=\exp(\pi\sqrt{-1}/c-\pi\sqrt{-1}/2)$. Then the equivalence
of~Conjecture~\ref{gaiotto}(b) is nothing but the Fundamental Local Equivalence
$\Rep_q(\SO(2n+1))\iso\mathrm{D}\modu_{c^{-1}}^{\on{U}(\bF),\chi,\lc}(\Gr_G)$ for the category of twisted
Whittaker $D$-modules on $\Gr_G$. On the other hand, ~Conjecture~\ref{gaiotto}(a) proposes
an equivalence  $\Rep_{q'}(\on{SOSp}(1|2n))\iso\mathrm{D}\modu_{c^{-1}}^{\on{U}(\bF),\chi,\lc}(\Gr_G)$.
Combining the two equivalences we obtain a purely algebraic equivalence
$\Rep_q(\SO(2n+1))\simeq\Rep_{q'}(\on{SOSp}(1|2n))$ that is nothing but the equivalence induced
by the twistor isomorphism $\dot\Psi$ between the integral forms of the modified
quantum (super) groups of~\cite[Theorem 4.3]{cflw}.
\end{rem}

\begin{rem}
  More generally, for arbitrary $k$, combining the equivalences of~Conjecture~\ref{gaiotto}(a,b),
  we arrive at an equivalence $\Rep_q(\on{SOSp}(2n+1|2k))\simeq\Rep_{q'}(\on{SOSp}(2k+1|2n))$,
  cf.~\cite[\S5.6.6]{mw}. It is induced by the isomorphism~\cite[Theorem 3.1, \S3.1]{xz} between the
  modified quantum supergroups.\footnote{We are grateful to A.~Tsymbaliuk for bringing~\cite{xz}
  to our attention.}
\end{rem}

\begin{rem}
  Let $\CS_e\subset\fsp(2n)$ denote the Slodowy slice through a nilpotent element $e$ of ``hook''
  Jordan type $(2n-2k,1^{2k})$. Then the symplectic variety
  $\Sp(2n)\times\CS_e$ is obtained by the Hamiltonian reduction
  $(T^*\Sp(2n)\times\BC^{2k})/\!\!/\on{U}_k$. It is a {\em hyperspherical} variety of the group
  $\Sp(2n)\times\Sp(2k)$ (i.e.\ the algebra of invariant functions
  $\BC[\Sp(2n)\times\CS_e]^{\Sp(2n)\times\Sp(2k)}$ is Poisson commutative; equivalently, a typical
  $\Sp(2n)\times\Sp(2k)$-orbit in $\Sp(2n)\times\CS_e$ is coisotropic), see~\cite[Proposition 2.2.3]{fu}.

    The twisted $S$-dual of
  $\Sp(2n)\times\CS_e\circlearrowleft\Sp(2n)\times\Sp(2k)$ (twisting along $\Sp(2n)$) is
  $\BC^{2n}\otimes\BC^{2k+1}\circlearrowleft\Sp(2n)\times\SO(2k+1)$ (the bifundamental representation
    of the even part of $\on{SOSp}(2k+1|2n)$ on $\mathfrak{osp}(2k+1|2n)_{\bar1}$). If the twisting
    is taken along $\Sp(2k)$, then the twisted $S$-dual is
  $\BC^{2n+1}\otimes\BC^{2k}\circlearrowleft\SO(2n+1)\times\Sp(2k)$ (the bifundamental representation
  of the even part of $\on{SOSp}(2n+1|2k)$ on $\mathfrak{osp}(2n+1|2k)_{\bar1}$).
\end{rem}

\subsection{Gaiotto conjectures for $\GL(K|N)$ revisited}
\label{glkn}
Let $0<K<N-1$ be some positive integers. We will reformulate~\cite[Conjecture 2.6.1]{bfgt}
similarly to~\S\ref{osp2n}. Namely, we consider the unipotent subgroup $U'_{M,N}\subset\GL_N$
of~\cite[Remark 2.4.1]{bfgt} for $M=K-1$. The $1_{M+1}=1_K$ block in the middle of the matrix
$\begin{pmatrix}U_r& {*} &{*}\\0&1_{M+1}& {*}\\0&0& U_s\end{pmatrix}$ of {\em loc.cit.}\ has
a row of $K$ matrix elements right above it; a column of $K$ matrix elements immediately to the
right of it, and a matrix element at the intersection of these row and column; $2K+1$ matrix
elements altogether. Annihilating all the other above-diagonal matrix elements of $U'_{K-1,N}$ we
obtain a surjective homomorphism $\phi'\colon U'_{K-1,N}\twoheadrightarrow\on{Heis}(\BC^{2K})$ onto the
Heisenberg group of a $2K$-dimensional symplectic vector space. This symplectic vector space is
equipped with a natural polarization $\BC^{2K}\cong\BC^K\oplus\BC^{K*}$ and an action of $\GL_K$,
and $\phi'$ is $\GL_K$-equivariant.

Furthermore, the character $\chi^{(r,s)}_{M,N}$ of~\cite[Remark 2.4.1]{bfgt} is the sum of certain
matrix elements of $U'_{M,N}$. Let $\bar\chi{}^{(r,s)}_{K-1,N}$ denote the partial sum excluding the
summands corresponding to the $2K+1$ matrix entries described in the previous paragraph. We view
$\bar\chi{}^{(r,s)}_{K-1,N}$ as a homomorphism to the center of the Heisenberg group:
$U'_{K-1,N}\to\BC\hookrightarrow\on{Heis}(\BC^{2K})$. We define a surjective homomorphism
$\phi\colon U'_{K-1,N}\twoheadrightarrow\on{Heis}(\BC^{2K})$ as $\phi:=\phi'+\bar\chi{}^{(r,s)}_{K-1,N}$.
It is still $\GL_K$-equivariant.

We consider the completed Weyl algebra $\CW(\bF^{2K})$. Given $c\in\BC^\times$ we consider the
category $\CalD\CW(\bF^{2K})\modu_{c^{-1}}:=\rm{D}\modu_{c^{-1}}(\Gr_{\GL_N})\otimes\CW(\bF^{2K})\modu$
and the category $\CalD\CW(\bF^{2K})\modu_{c^{-1}}^{\GL(K,\bO)\ltimes U'_{K-1,N}(\bF),\lc}$ of locally compact
$\GL(K,\bO)\ltimes U'_{K-1,N}(\bF)$-equivariant objects (here $U'_{K-1,N}(\bF)$ is understood to
act on $\CW(\bF^{2K})$ via $\phi$).

\begin{conj}
  For $c\not\in\BQ^\times,\ q=\exp(\pi\sqrt{-1}/c)$, the categories
  \[\CalD\CW(\bF^{2K})\modu_{c^{-1}}^{\GL(K,\bO)\ltimes U'_{K-1,N}(\bF),\lc}\ \on{and}\
  \Rep_q(\GL(K|N))\]
  are equivalent as braided tensor categories, and this equivalence is compatible with the
  tautological $t$-structures.\footnote{Proved in~\cite[Theorem 10.1.8]{ty2}.}
\end{conj}

\begin{rem}
  Let $\CS_e\subset\fgl_N$ denote the Slodowy slice through a nilpotent element $e$ of ``hook''
  Jordan type $(N-K,1^K)$. Then the symplectic variety $\GL_N\times\CS_e$ is obtained by
  the Hamiltonian reduction $(T^*\GL_N\times\BC^{2K})/\!\!/U'_{K-1,N}$. Alternatively, it can
  be obtained by the Hamiltonian reduction $(T^*\GL_N)/\!\!/(U'_{K,N},\chi_{K,N}^{(r,s)})$.
  It is a hyperspherical variety of the group $\GL_N\times\GL_K$, see~\cite[Proposition 2.2.2]{fu}.

  The $S$-dual of $\GL_N\times\CS_e\circlearrowleft\GL_N\times\GL_K$ is
  $T^*\Hom(\BC^N,\BC^K)\circlearrowleft\GL_N\times\GL_K$ (the representation of the even part
  of $\GL(N|K)$ on $\fgl(N|K)_{\bar1}$).
\end{rem}

\subsection{Exceptional Lie superalgebra $\ff(4)$}
\label{f4}
We consider a nilpotent element $e\in\fsp(6)$ of Jordan type $(3,3)$, so that $e$ lies in a
$14$-dimensional nilpotent orbit. We fix a maximal reductive subgroup $\PGL(2)$ in the centralizer
$Z_{\on{PSp}(6)}(e)$ of $e$. We choose an $\fsl_2$-triple $(e,h,f)$ in $\fsp(6)$.
The adjoint action of $h$ on $\fsp(6)$ equips it with a grading, and $\fsp(6)_{-1}=0$ ($e$ is an
even nilpotent element), while $\fu=\fsp(6)_{\leq-2}$ is the 7-dimensional nilpotent radical of
the parabolic subalgebra stabilizing an isotropic 2-plane in $\BC^6$. We denote by
$\on{U}\subset\on{PSp}(6)$ the unipotent subgroup with Lie algebra $\fu$. It is normalized by
$\PGL(2)\subset Z_{\on{PSp}(6)}(e)$. Finally, the Killing pairing with $e$ gives rise to a character
$\chi^0$ of $\fu$ and the same named character $\chi^0$ of $\on{U}$.

Extending the scalars to $\bF$ we obtain a character $\chi^0\colon\on{U}(\bF)\to\bF$, and we set
$\chi(u):=\Res_{t=0}\chi^0(u)dt\colon\on{U}(\bF)\to\BG_a$. Given $c\in\BC^\times$ we consider the category
$\mathrm{D}\modu_{c^{-1}}(\Gr_{\on{PSp}(6)})$ and the category
$\mathrm{D}\modu_{c^{-1}}(\Gr_{\on{PSp}(6)})^{\PGL(2,\bO)\ltimes \on{U}(\bF),\chi,\lc}$ of locally compact
$(\PGL(2,\bO)\ltimes \on{U}(\bF),\chi)$-equivariant objects.


On the dual side, we consider the quantum group $U_q(\ff(4)),\ q=\exp(\pi\sqrt{-1}/c)$,
and we denote by $\Rep_q(\on{F}(4))$ the dg-category of finite dimensional complexes of
$U_q(\ff(4))$-modules.

\begin{conj}
  \label{gaiotto f4}
   For $c\not\in\BQ^\times$, the categories
  \[\mathrm{D}\modu_{c^{-1}}(\Gr_{\on{PSp}(6)})^{\PGL(2,\bO)\ltimes \on{U}(\bF),\chi,\lc}\ \on{and}\
  \Rep_q(\on{F}(4))\]
  are equivalent as braided tensor categories, and this equivalence is compatible with the
  tautological $t$-structures.\footnote{Proved in~\cite[Corollary 4.4.4]{fty} for transcendental $q$.}
\end{conj}

\begin{rem}
  Let $\CS_e\subset\fsp(6)$ denote the Slodowy slice through $e$. Then the symplectic variety
  $\on{PSp}(6)\times\CS_e$ is obtained by the Hamiltonian reduction
  $(T^*\on{PSp}(6))/\!\!/(\on{U},\chi^0)$. It is a hyperspherical variety of the group
  $\on{PSp}(6)\times\PGL(2)$, see~\cite[Proposition 2.3.1]{fu}. The $S$-dual of
  $\on{PSp}(6)\times\CS_e\circlearrowleft\on{PSp}(6)\times\PGL(2)$ is
  $\BC^8\otimes\BC^2\circlearrowleft\on{Spin}(7)\times\SL(2)$ (the bispinor representation of the
  even part of $\on{F}(4)$ on $\ff(4)_{\bar1}$).
\end{rem}

\begin{rem}
  There is an order~4 outer automorphism $\sigma$ of the simply connected group $E_6$ such that its
  fixed points coincide with the even part of $\on{F}(4)$, and the $\sqrt{-1}$-eigenspace of
  $\sigma$ on ${\mathfrak e}_6$ coincides with $\ff(4)_{\bar1}$,
  see~\cite[row 24 of the table in~\S9]{v}.\footnote{We are grateful to A.~Elashvili for this remark.}
\end{rem}

\subsection{Exceptional Lie superalgebra $\fg(3)$}
\label{g3}
We consider a nilpotent element $e\in\fg_2$ corresponding to a short root vector, so that
$e$ lies in the $8$-dimensional nilpotent orbit.
We fix a maximal reductive
subgroup $\SL(2)$ in the centralizer of $e$. We choose an $\fsl_2$-triple $(e,h,f)$ in $\fg_2$.
The adjoint action of $h$ on $\fg_2$ equips it with a grading, and $(\fg_2)_{-1}\simeq\BC^2$
carries a canonical symplectic form. The $5$-dimensional nilpotent Lie algebra
$\fu=(\fg_2)_{\leq-1}$ projects onto $\mathfrak{Heis}((\fg_2)_{-1})$. We denote by $\on{U}\subset \on{G}_2$
the unipotent subgroup with Lie algebra $\fu$. It is normalized by $\SL(2)\subset Z_{\on{G}_2}(e)$.
Note that $\SL(2)\ltimes \on{U}$ is the derived group of a parabolic subgroup of $\on{G}_2$
(the stabilizer of the highest weight line in the 7-dimensional representation of $\on{G}_2$).

As in~\S\ref{weyl}, we consider the completed Weyl algebra $\CW(\bF^2)$. There is a twisted
action $\mathrm{D}\modu_{-1/2}(\SL(2,\bF))\circlearrowright\CW(\bF^2)\modu$.
Given $c\in\BC^\times$ we consider the category
$\CalD\CW(\bF^2)\modu_{c^{-1}}:=\mathrm{D}\modu_{c^{-1}}(\Gr_{\on{G}_2})\otimes\CW(\bF^2)\modu$,
and the category $\CalD\CW(\bF^2)\modu_{c^{-1}}^{\SL(2,\bO)\ltimes \on{U}(\bF),\lc}$ of locally compact
$\SL(2,\bO)\ltimes \on{U}(\bF)$-equivariant objects.

On the dual side, we consider the quantum group $U_q(\fg(3)),\ q=\exp(\pi\sqrt{-1}/c)$,
and we denote by $\Rep_q(\on{G}(3))$ the dg-category of finite dimensional complexes of
$U_q(\fg(3))$-modules.

\begin{conj}
  \label{gaiotto g3}
   For $c\not\in\BQ^\times$, the categories
  \[\CalD\CW(\bF^2)\modu_{c^{-1}}^{\SL(2,\bO)\ltimes \on{U}(\bF),\lc}\ \on{and}\
  \Rep_q(\on{G}(3))\]
  are equivalent as braided tensor categories, and this equivalence is compatible with the
  tautological $t$-structures.
\end{conj}

\begin{rem}
  Let $\CS_e\subset\fg_2$ denote the Slodowy slice through $e$. Then the symplectic variety
  $\on{G}_2\times\CS_e$ is obtained by the Hamiltonian reduction
  $(T^*\on{G}_2\times\BC^2)/\!\!/\on{U}$. It is a hyperspherical variety of the group
  $\on{G}_2\times\SL(2)$, see~\cite[Proposition 2.4.1]{fu}.
  The twisted $S$-dual of
  $\on{G}_2\times\CS_e\circlearrowleft\on{G}_2\times\SL(2)$ is
  $\BC^7\otimes\BC^2\circlearrowleft\on{G}_2\times\SL(2)$ (the bifundamental representation of the
  even part of $\on{G}(3)$ on $\fg(3)_{\bar1}$).
\end{rem}

\subsection{Kostka polynomials}

\subsubsection{Case of $\osp(2n+1|2n)$}
Following~\cite[\S4]{gl}, we consider the {\em mixed} Borel subalgebra $\fb\subset\bg=\osp(V_0|V)$
with the set of positive simple roots
\[\{\varepsilon_1-\delta_1,\delta_1-\varepsilon_2,\varepsilon_2-\delta_2,\ldots,
\varepsilon_{n-1}-\delta_{n-1},\delta_{n-1}-\varepsilon_n,\varepsilon_n-\delta_n,\delta_n\}.\]
All the above simple roots are odd, and all but the last one are isotropic.
The odd part $\fn_{\bar1}$ of the nilpotent radical $\fn$ of $\fb$ has Cartan eigenvalues
\[R_{\bar1}^+=\{\varepsilon_i+\delta_j\}_{1\leq i,j\leq n}\cup\{\varepsilon_i-\delta_j\}_{1\leq i\leq j\leq n}
\cup\{\delta_i-\varepsilon_j\}_{1\leq i<j\leq n}\cup\{\delta_i\}_{1\leq i\leq n}.\]
Following~\cite[Definition 3.3.1]{bft}, for $\alpha\in\ft^*\oplus\ft_0^*$ (notation of~\S\ref{sosp})
we consider a polynomial $L_\alpha(q):=\sum p_dq^d$ where $p_d$ is the number of unordered partitions
of $\alpha$ into a sum of $d$ elements of $R_{\bar1}^+$. We say that
$\Lambda_0^+\times\Lambda_1^+\ni(\lambda_0,\lambda_1)\geq(\mu_0,\mu_1)$ if
$(\lambda_0,\lambda_1)-(\mu_0,\mu_1)\in\BN\langle R_{\bar1}^+\rangle$. We set
\[K_{(\lambda_0,\lambda_1),(\mu_0,\mu_1)}(q):=\sum_{w_0\in W_0,\ w\in W}(-1)^{w_0}(-1)^w
L_{(w_0(\lambda_0+\rho_0)-\rho_0-\mu_0,w(\lambda_1+\rho_1)-\rho_1-\mu_1)}(q),\]
notation of~\S\ref{sosp}.

Recall the notation of~\S\ref{2.6}, especially the paragraph before~Proposition~\ref{irreducibles}.
Given $(\lambda_0,\lambda_1)\in\Lambda_0^+\times\Lambda_1^+$, we will view the irreducible object
$\IC_{\lambda_0}*E_0*\IC_{\lambda_1}\simeq\IC_{(\lambda_0,\lambda_1)}\in\CalD\CW\modu^{G_\bO,\lc,\heartsuit}$
as a $G_\bO$-equivariant
$D$-module on $(V_\bF/V_\bO)\times\Gr_G$. Given $(\mu_0,\mu_1)\in\Lambda_0^+\times\Lambda_1^+$, we
are interested in the stalks of $\IC_{(\lambda_0,\lambda_1)}$ at the relevant $G_\bO$-orbit
$\BO_{(\mu_0,\mu_1)}\subset(V_\bF/V_\bO)\times\Gr_G$. The following theorem is 
similar to~\cite[Theorem 3.3.5]{bft}.

\begin{thm}
  \label{kostka}
\textup{(a)} If a relevant $G_\bO$-orbit $\BO_{(\mu_0,\mu_1)}\subset(V_\bF/V_\bO)\times\Gr_G$ lies in the
closure of a relevant $G_\bO$-orbit $\BO_{(\lambda_0,\lambda_1)}$, then $(\lambda_0,\lambda_1)\geq(\mu_0,\mu_1)$.
Conversely, if $(\lambda_0,\lambda_1)\geq(\mu_0,\mu_1)$, and all the entries of $\mu_0$ are strictly positive,
then $\BO_{(\mu_0,\mu_1)}$ lies in the closure of $\BO_{(\lambda_0,\lambda_1)}$.

\textup{(b)} If all the entries of $\mu_0$ are strictly positive, then
\[q^{-\dim\BO_{(\mu_0,\mu_1)}}K_{(\lambda_0,\lambda_1),(\mu_0,\mu_1)}(q^{-1})=
\sum_i\dim(\IC_{(\lambda_0,\lambda_1)})^{-i}_{\BO_{(\mu_0,\mu_1)}}q^{-i},\]
  (the Poincar\'e polynomial of the $\IC_{(\lambda_0,\lambda_1)}$-stalks at the orbit $\BO_{(\mu_0,\mu_1)}$).\footnote{We are grateful to
the referee for his correction of our mistake in the earlier version of this note.}
\end{thm}

\begin{proof}
  The argument is similar to the proof of~\cite[Theorem 3.3.5]{bft} with the following caveat.
  We consider the positive Borel subgroups $B_0\subset\SO(V_0)$, $B_1\subset\Sp(V)$, and the vector
  bundle $\widetilde{\CN}_{\bar1}=(\SO(V_0)\times\Sp(V))\stackrel{B_0\times B_1}{\times}\fn_{\bar1}$ over the product $\CB_0\times\CB_1$
  of the flag varieties of $\SO(V_0)$ and $\Sp(V)$. This vector bundle is a resolution of singularities of the nilcone
  $\CN_{\bar1}\subset V_0\otimes V$. The similar resolutions $\widetilde{\CN}{}^N_{\bar1}$ were used in the proof of~\cite[Theorem 3.3.5]{bft}.
  However, while the canonical class of $\widetilde{\CN}{}^N_{\bar1}$ was equivariantly trivial (see the proof of~\cite[Corollary 3.3.3]{bft}),
  the canonical class $\bomega_{\widetilde{\CN}_{\bar1}}$ is equivariantly isomorphic to
  $\CO_{\widetilde{\CN}_{\bar1}}(\sum_{i=1^n}(\varepsilon_i-\delta_i))$ (pullback of the line bundle from $\CB_0\times\CB_1$).
  Equivalently, the sum of all positive even roots of the mixed Borel subalgebra minus the sum of all positive odd roots of the
  mixed Borel subalgebra equals $2\rho=\sum_{i=1}^n(\delta_i-\varepsilon_i)$.

  Now if all the entries of $\mu_0$ are strictly positive, the higher cohomology of $\widetilde{\CN}_{\bar1}$ with coefficients in
  $\CO_{\widetilde{\CN}_{\bar1}}(\mu_0,\mu_1)$ vanish just as in {\em loc.cit.} Under the same condition, $(\mu_0,\mu_1)+\rho$ is even-dominant.
  Now the argument in the proof of~\cite[Theorem 3.3.5(c,d)]{bft} goes through.
\end{proof}

\begin{rem}
  (a) We expect that in fact, the conditions $\BO_{(\mu_0,\mu_1)}\subset\ol\BO_{(\lambda_0,\lambda_1)}$, and $(\lambda_0,\lambda_1)\geq(\mu_0,\mu_1)$
  are equivalent (without the extra condition of strict positivity of all entries of $\mu_0$).
  
  (b) However, the condition of strict positivity of all entries of $\mu_0$ in~Theorem~\ref{kostka}(b) cannot be weakened. Take $n=1$,
  $(\lambda_0,\lambda_1)=(0,\delta_1)$, $(\mu_0,\mu_1)=(0,0)$. The corresponding Kostka polynomial is $q$, and the LHS
  of~Theorem~\ref{kostka}(b) is $q^{-1}$. But the RHS is~$0$, as it was explained in~\S\ref{generators} that $\IC_{(0,\delta_1)}$
  is a {\em clean} extension from the corresponding orbit.
\end{rem}

\subsubsection{Case of $\ff(4)$}
In the setup of~\S\ref{f4}, let us consider the limit case $c^{-1}\to0$, and the category
$\mathrm{D}\modu(\Gr_{\on{PSp}(6)})^{\PGL(2,\bO)\ltimes\on{U}(\bF),\chi}$
({\em untwisted} $D$-modules), along with the category
$\mathrm{D}\modu(\Gr_{\on{PSp}(6)})^{\PGL(2,\bO)\ltimes\on{U}(\bF),\chi,\lc}$
of locally compact $\PGL(2,\bO)\ltimes \on{U}(\bF)$-equivariant objects.
Similarly to~Theorem~\ref{main}(d), we expect that there is a monoidal equivalence from the
category of representations of the degenerate supergroup $D^b\Rep(\ul{\on{F}}(4))$ to the
category $S\mathrm{D}\modu(\Gr_{\on{PSp}(6)})^{\PGL(2,\bO)\ltimes\on{U}(\bF),\chi,\lc}:=
\mathrm{D}\modu(\Gr_{\on{PSp}(6)})^{\PGL(2,\bO)\ltimes\on{U}(\bF),\chi,\lc}\otimes_{\on{Vect}}\on{SVect}$.

We can view irreducible objects of
$\mathrm{D}\modu(\Gr_{\on{PSp}(6)})^{\PGL(2,\bO)\ltimes\on{U}(\bF),\chi,\lc,\heartsuit}$
as irreducible
$\PGL(2,\bO)\ltimes\on{U}(\bF)$-equivariant $D$-modules on $\Gr_{\on{PSp}(6)}$ supported
on certain {\em relevant} $\PGL(2,\bO)\ltimes\on{U}(\bF)$-orbits, similarly to~\S\ref{2.6}. We expect that as
in~Theorem~\ref{kostka}(a), the adjacency relation of relevant orbits is governed by a choice of a
certain Borel subalgebra in ${\mathfrak f}(4)$. Namely, in notation of~\cite[\S2.18]{fss} (cf.\
also~\cite[Exercise 4.7.12]{m}\footnote{There is a typo in~\cite[Exercise 4.7.11]{m}: the definitions
of simple roots $\alpha_1,\alpha_3$ of $\on{Spin}(7)$ should be swapped:
$\alpha_1=\varepsilon_1-\varepsilon_2,\ \alpha_3=\varepsilon_3$.})
the Borel subalgebra $B_4$ in question has positive simple roots
$\{\frac12(\delta+\varepsilon_1-\varepsilon_2-\varepsilon_3),
\frac12(\delta-\varepsilon_1+\varepsilon_2+\varepsilon_3),
\frac12(-\delta+\varepsilon_1-\varepsilon_2+\varepsilon_3),\varepsilon_2-\varepsilon_3\}$.

Let $B'\subset\on{Spin}(7),\ B''\subset\SL(2)$ be the Borel subgroups whose Lie algebras are
contained in the Borel subalgebra $B_4\subset\ff(4)$. Let $\fn_{\bar1}$ be the odd part of the
nilpotent radical of $B_4$. Let $\fn'_{\bar1}\subset\fn_{\bar1}$ be the 6-dimensional subspace
spanned by the root vectors of weights
\begin{multline*}
\frac12(\delta+\varepsilon_1-\varepsilon_2-\varepsilon_3),
\frac12(\delta-\varepsilon_1+\varepsilon_2+\varepsilon_3),
\frac12(\delta+\varepsilon_1-\varepsilon_2+\varepsilon_3),\\
\frac12(\delta+\varepsilon_1+\varepsilon_2-\varepsilon_3),
\frac12(\delta+\varepsilon_1+\varepsilon_2+\varepsilon_3),
\frac12(-\delta+\varepsilon_1+\varepsilon_2+\varepsilon_3).
\end{multline*}

Then the natural action $B'\times B''\circlearrowright\fn'_{\bar1}$ extends
to the action $P'\times B''\circlearrowright\fn'_{\bar1}$ where $P'\supset B'$ is the subminimal
parabolic subgroup corresponding to the (long) middle simple root $\varepsilon_2-\varepsilon_3$
of $\on{Spin}(7)$.
In other words, $\on{Spin}(7)/P'$ is the isotropic flag space $\CF\ell(1,3,7)$ of the
7-dimensional space equipped with the symmetric bilinear form preserved by $\on{Spin}(7)$.
Let $\widetilde\CN_{\bar1}$ be the vector bundle over
$(\on{Spin}(7)/P')\times(\SL(2)/B'')=\CF\ell(1,3,7)\times\BP^1$
associated to the representation $P'\times B''\circlearrowright\fn'_{\bar1}$. We have a natural
morphism $\widetilde\CN_{\bar1}\to\CN_{\bar1}$ to the odd nilpotent cone of $\ff(4)$. This is a
particular case of Hesselink's desingularization~\cite{h} (in particular, both
$\widetilde\CN_{\bar1}$ and $\CN_{\bar1}$ have dimension~$15$, cf.~\cite[row 4 of Table 6]{e}).
We are grateful to A.~Elashvili and M.~Jibladze for this observation.

\subsubsection{Case of $\fg(3)$}
In the setup of~\S\ref{g3}, let us consider the limit case $c^{-1}\to0$, and the category
$\CalD'\CW(\bF^2)\modu:=\mathrm{D}\modu(\Gr_{\on{G}_2})\otimes\CW(\bF^2)\modu$ ({\em untwisted}
$D$-modules), along with
the category $\CalD'\CW(\bF^2)\modu^{\SL(2,\bO)'\ltimes \on{U}(\bF),\lc}$ of locally compact
$\SL(2,\bO)\ltimes \on{U}(\bF)$-equivariant objects.
Similarly to~Theorem~\ref{main}(d), we expect that there is a monoidal equivalence from the
category of representations of the degenerate supergroup $D^b\Rep(\ul{\on{G}}(3))$ to the
category $S\CalD'\CW(\bF^2)\modu^{\SL(2,\bO)'\ltimes \on{U}(\bF),\lc}:=
\CalD'\CW(\bF^2)\modu^{\SL(2,\bO)'\ltimes \on{U}(\bF),\lc}\otimes_{\on{Vect}}\on{SVect}$.

Similarly to~\S\ref{2.6}, we can view
irreducible objects of $\CalD'\CW(\bF^2)\modu^{\SL(2,\bO)'\ltimes \on{U}(\bF),\lc,\heartsuit}$ as irreducible
$\SL(2,\bO)\ltimes\on{U}(\bF)$-equivariant $D$-modules on $(\bF^2/\bO^2)\times\Gr_{\on{G}_2}$ supported
on certain {\em relevant} $\SL(2,\bO)\ltimes\on{U}(\bF)$-orbits. We expect that as
in~Theorem~\ref{kostka}(a), the adjacency relation of relevant orbits is governed by a choice of a
certain Borel subalgebra in $\fg(3)$. Namely, in notation of~\cite[\S2.19]{fss} (cf.\
also~\cite[Exercise 4.7.10]{m}) the Borel subalgebra $B_4$ in question has positive simple roots
$\{\delta-\varepsilon_1,-\delta+\varepsilon_2,\varepsilon_1\}$.

Let $B'\subset\on{G}_2,\ B''\subset\SL(2)$ be the Borel subgroups whose Lie algebras are
contained in the Borel subalgebra $B_4\subset\fg(3)$. Let $\fn_{\bar1}$ be the odd part of the
nilpotent radical of $B_4$. Then the natural action $B'\times B''\circlearrowright\fn_{\bar1}$ extends
to the action $P'\times B''\circlearrowright\fn_{\bar1}$ where $P'\supset B'$ is the
parabolic subgroup whose Lie algebra contains the negative simple {\em short} root space.
Let $\widetilde\CN_{\bar1}$ be the vector bundle over
$(\on{G}_2/P')\times(\SL(2)/B'')=(\on{G}_2/P')\times\BP^1$
associated to the representation $P'\times B''\circlearrowright\fn_{\bar1}$. We have a natural
morphism $\widetilde\CN_{\bar1}\to\CN_{\bar1}$ to the odd nilpotent cone of $\fg(3)$. This is a
particular case of Hesselink's desingularization~\cite{h} (in particular, both
$\widetilde\CN_{\bar1}$ and $\CN_{\bar1}$ have dimension~$13$, cf.~\cite[row 8 of Table 6]{e}).
We are grateful to A.~Elashvili and M.~Jibladze for this observation.

\end{document}